\documentclass[aop,preprint]{imsart}

\RequirePackage{amsthm,amsmath}
\RequirePackage{natbib}
\RequirePackage[colorlinks,citecolor=blue,urlcolor=blue]{hyperref}

\usepackage[a4paper]{geometry}

\usepackage{textcomp}
\usepackage{bm}
\usepackage{amstext}
\usepackage{amssymb}
\usepackage{esint}

\startlocaldefs

\numberwithin{equation}{section}
\numberwithin{figure}{section}
\theoremstyle{plain}
\newtheorem{thm}{\protect\theoremname}[section]
  \theoremstyle{definition}
  \newtheorem{example}[thm]{\protect\examplename}
  \theoremstyle{definition}
  \newtheorem{condition}[thm]{\protect\conditionname}
  \theoremstyle{definition}
  \newtheorem{defn}[thm]{\protect\definitionname}
 \theoremstyle{definition}
 \newtheorem*{defn*}{\protect\definitionname}
  \theoremstyle{plain}
  \newtheorem*{assumption*}{\protect\assumptionname}
  \theoremstyle{remark}
  \newtheorem*{rem*}{\protect\remarkname}
  \theoremstyle{remark}
  
  \theoremstyle{plain}
  \newtheorem{prop}[thm]{\protect\propositionname}
  \theoremstyle{plain}
  \newtheorem{lem}[thm]{\protect\lemmaname}
  \theoremstyle{plain}
  \newtheorem{cor}[thm]{\protect\corollaryname}
  \theoremstyle{remark}
  \newtheorem{claim}[thm]{\protect\claimname}

\newcommand{\eps}{\varepsilon}  
  
\newcommand{\vf}{\varphi}
\newcommand{\al}{\alpha}
\newcommand{\be}{\beta}

\newcommand{\si}{\sigma}

\newcommand{\tht}{\theta}
\newcommand{\om}{\omega}
\newcommand{\vp}{\varpi}

\newcommand{\md}{\mathrm{d}}   
\newcommand{\vd}{\,\md}
\newcommand{\me}{\mathrm{e}}   

\newcommand{\pd}{\partial}   
\newcommand{\esssup}{\mathop{\mathrm{ess}\sup}}

\newcommand{\R}{\mathbf{R}}   
\newcommand{\Nat}{\mathbf{Z}_+}
\newcommand{\cfun}{\mathcal{C}}

\newcommand{\bO}{\mathcal{O}}

\newcommand{\PS}{\Omega}  
\newcommand{\E}{\mathbb{E}}  
\newcommand{\Prob}{\mathbb{P}}  
\newcommand{\BM}{w}  

\newcommand{\m}{\ell} 

\newcommand{\bbr}{|\hspace{-0.3ex}|\hspace{-0.3ex}|}
\newcommand{\lbr}{[\hspace{-0.33ex}[}
\newcommand{\rbr}{]\hspace{-0.33ex}]}
\newcommand{\Dom}{\mathcal{O}}
\newcommand{\hold}{\delta}

  \providecommand{\assumptionname}{Assumption}
  \providecommand{\claimname}{Claim}
  \providecommand{\conditionname}{Condition}
  \providecommand{\corollaryname}{Corollary}
  \providecommand{\definitionname}{Definition}
  \providecommand{\examplename}{Example}
  \providecommand{\lemmaname}{Lemma}
  \providecommand{\propositionname}{Proposition}
  \providecommand{\remarkname}{Remark}
\providecommand{\theoremname}{Theorem}

\endlocaldefs

\begin{document}

\begin{frontmatter}

\title{
Stochastic continuity of random fields governed by a system of stochastic PDEs
}
\runtitle{Stochastic parabolic systems}

\begin{aug}
\author{\fnms{Kai} \snm{Du}
\ead[label=e1]{kdu@fudan.edu.cn}}
\author{\fnms{Jiakun} \snm{Liu}
\ead[label=e2]{jiakunl@uow.edu.au}}
\and
\author{\fnms{Fu} \snm{Zhang}\thanks{F. Zhang was partially supported by the National Natural Science Foundation of China (Grants \#11701369).}
\ead[label=e3]{fugarzhang@163.com}
\ead[label=u1,url]{http://www.foo.com}}

\runauthor{K. Du, J. Liu and F. Zhang}


\address{
	K. Du\\
	Shanghai Center for Mathematical Sciences\\
	Fudan University\\
	Shanghai 200433\\
	China\\
	\printead{e1}
}

\address{
J. Liu\\
School of Mathematics and Applied Statistics\\
University of Wollongong\\
Wollongong NSW 2522\\
Australia\\
\printead{e2}
}

\address{
F. Zhang\\
College of Science\\
University of Shanghai for Science and Technology\\
Shanghai 200093\\
China\\
\printead{e3}
}
\end{aug}

\begin{abstract}
This paper constructs a solvability theory for a system of stochastic
partial differential equations. On account of the Kolmogorov continuity
theorem, solutions are looked for in certain H\"{o}lder-type classes in
which a random field is treated as a space-time function taking values
in $L^{p}$-space of random variables. A modified stochastic parabolicity
condition involving $p$ is proposed to ensure the finiteness of the
associated norm of the solution, which is showed to be sharp by examples.
The Schauder-type estimates and the solvability theorem are proved. \end{abstract}



\end{frontmatter}

\section{Introduction}

Random fields governed by systems of stochastic partial differential
equations (SPDEs) have been used to model many physical phenomena in
random environments such as the motion of a random string, stochastic
fluid mechanic, the precessional motion of magnetisation with random
perturbations, and so on; specific models can be founded in \cite{funaki1983random,mueller2002hitting,mikulevicius2004stochastic,hairer2006ergodicity,brzezniak2013weak,da2014stochastic}
and references therein. 
This paper concerns the smoothness properties
of the random field 
\[
{\bm{u}}=(u^{1},\dots,u^{N})':\R^{d}\times[0,\infty)\times\PS\to\R^{N}
\]
described by the following linear system of SPDEs: 
\begin{equation}
\md u^{\alpha}=\big(a_{\alpha\beta}^{ij}\partial_{ij}u^{\beta}+b_{\alpha\beta}^{i}\partial_{i}u^{\beta}+c_{\alpha\beta}u^{\beta}+f_{\alpha}\big)\vd t+\bigl(\sigma_{\alpha\beta}^{ik}\partial_{i}u^{\beta}+\nu_{\alpha\beta}^{k}u^{\beta}+g_{\alpha}^{k}\bigr)\vd\BM_{t}^{k},\label{eq:main}
\end{equation}
where $\{\BM^{k}\}$ are countable independent Wiener processes defined
on a filtered complete probability space $(\Omega,\mathcal{F},(\mathcal{F}_{t})_{t\ge 0},\Prob)$,
and Einstein's summation convention is used with 
\[
i,j=1,2,\dots,d;\text{\quad}\alpha,\beta=1,2,\dots,N;\quad k=1,2,\dots,
\]
and the coefficients and free terms are all random fields. 
Considering infinitely many Wiener processes enables us to treat systems driven by space-time white noise (see~\cite{krylov1999analytic}).
Regularity
theory for system \eqref{eq:main} can not only directly apply to
some concrete models, see for example \cite{zakai1969,funaki1983random,walsh1986introduction,mueller2002hitting},
but also provide with important estimates for solutions of suitable
approximation to nonlinear systems in the literature such as \cite{krylov1997spde,mikulevicius2012unbiased,da2014stochastic}
and references therein.

The literature dedicated to SPDEs (not systems)
is quite extensive and fruitful. In the framework of Sobolev spaces,
a complete $L^{p}$-theory ($p\ge2$) has been developed,
see \cite{pardoux1975these,krylov1977cauchy,krylov1996l_p,krylov1999analytic,van2012stochastic,chen2014lp}
and references therein. 
However,
the $L^{p}$-theory for systems of SPDEs is far from complete, though it has been fully solved for $p=2$ by \cite{Kim2013w},
and for $p>2$ some specific results were obtained by \cite{mikulevicius2001note,mikulevicius2004stochastic,kim2013note},
where the matrices $\sigma^{ik}=[\sigma_{\alpha\beta}^{ik}]_{N\times N}$
were diagonal or nearly diagonal. The smoothness properties of random fields follow
from Sobolev's embedding theorem in this framework.

The present paper investigates the regularity of random fields
from another aspect prompted by Kolmogorov's continuity theorem.
This theorem gives mild conditions under which a random fields has
a continuous modification, and the point is to derive appropriate
estimates on $L^{p}$-moments of increments of the random field. This
boosts an idea that considers a random field to be a function of $(x,t)$
taking values in the space $L_{\omega}^{p}:=L^{p}(\Omega)$ and introduces
appropriate $L_{\omega}^{p}$-valued H\"{o}lder classes as the working
spaces, for instance, the basic space used in \cite{rozovskiui1975stochastic,mikulevicius2000cauchy,du2015cauchy}
and also in the present paper defined to be the set of all jointly
measurable random fields $u$ such that
\[
\|u\|_{\cfun_{p}^{\hold}}:=\bigg[\sup_{t,\,x}\E|u(x,t)|^{p}+\sup_{t,\,x\neq y}\frac{\E|u(x,t)-u(y,t)|^{p}}{|x-y|^{\hold p}}\bigg]^{\frac{1}{p}}<\infty
\]
with some constants $\hold\in(0,1)$ and $p\in[2,\infty)$. Each random
field in this space $\cfun_{p}^{\hold}$ is stochastically continuous
in space, and if $\hold p>d$ it has a modification H\"{o}lder continuous
in space by Kolmogorov's theorem.

For the Cauchy problem for parabolic SPDEs (not systems), a $C^{2+\hold}$-theory was once
an open problem proposed by \cite{krylov1999analytic}; based on the H\"{o}lder class $\cfun_{p}^{\hold}$ it was partially
addressed by \cite{mikulevicius2000cauchy} and generally
solved by \cite{du2015cauchy,du2016schauder} very recently. 
They proved that, \emph{under natural conditions
on the coefficients, the solution $u$ and its derivatives $\partial u$
and $\partial^{2}u$ belong to the class $\mathcal{C}_{p}^{\hold}$
if $f$, $g$ and $\pd g$ belong to this space}; \cite{du2015cauchy}
further obtained the H\"{o}lder continuity in time of $\partial^{2}u$.
The main results of the theory are sharp in that they could not be
improved under the same assumptions. Extensions to the Cauchy\textendash Dirichlet
problem of SPDEs can be found in \cite{mikulevicius2003cauchy,mikulevicius2006cauchy},
and for more related results, we refer the reader to, for
instance, \cite{chow1994stochastic,bally1995approximation,tang2016cauchy}.
Nevertheless, $C^{2+\hold}$-theory for systems of SPDEs is not known
in the literature.

The purpose of this paper is to construct such a $C^{2+\hold}$-theory
for systems of type \eqref{eq:main} under mild conditions. 
Like the situation in the $L^p$ framework this extension is also nontrivial as some new features emerge in the system of SPDEs comparing with single equations. 
It is well-known that the well-posedness of a second order SPDE is usually
guaranteed by certain coercivity conditions. For system \eqref{eq:main},
\cite{Kim2013w} recently obtained $W_{2}^{n}$-solutions under the following algebraic condition:
there is a constant $\kappa>0$ such that
\begin{equation}
\big(2a_{\alpha\beta}^{ij}-\sigma_{\gamma\alpha}^{ik}\sigma_{\gamma\beta}^{jk}\big)\xi_{i}^{\alpha}\xi_{j}^{\beta}\ge\kappa|\xi|^{2}\quad\forall\,\xi\in\R^{d\times N}.\label{eq:parabolic}
\end{equation}
Although it is a natural extension of the strong ellipticity condition for
PDE systems ($\sigma\equiv0$, see for example \cite{schlag1996schauder})
and of the stochastic parabolicity condition for SPDEs ($N=1$, see
for example \cite{krylov1999analytic}), the
following example constructed by \cite{kim2013note} reveals
that condition \eqref{eq:parabolic} is not sufficient to ensure the
finiteness of $L_{\om}^{p}$-norm of the solution of some system even
the given data are smooth, and some structure condition stronger than \eqref{eq:parabolic}
is indispensable to establish a general $L^{p}$ or $C^{2+\hold}$
theory for systems of type \eqref{eq:main}.
\begin{example}
\label{exa:counterexam}Let $d=1$, $N=2$ and $p>2$. Consider the following
system: 
\begin{equation}
\bigg\{\begin{aligned}\md u^{(1)} & =u_{xx}^{(1)}\vd t-\mu u_{x}^{(2)}\vd\BM_{t},\\
\md u^{(2)} & =u_{xx}^{(2)}\vd t+\mu u_{x}^{(1)}\vd\BM_{t}
\end{aligned}
\label{eq:ceg}
\end{equation}
with the initial data 
\[
u^{(1)}(x,0)=\me^{-\frac{x^{2}}{2}},\quad u^{(2)}(x,0)=0,
\]
where $\mu$ is a given constant. In this case, condition \eqref{eq:parabolic}
reads $\mu^{2}<2$, but we will see that this is not sufficient to
ensure the finiteness of $\E|\bm{u}(x,t)|^{p}$ with $p>2$. Set $v=u^{(1)}+\sqrt{-1}u^{(2)}$,
and the above system turns to a single equation: 
\begin{equation}
\md v=v_{xx}\vd t+\sqrt{-1}\mu v_{x}\vd\BM_{t}\label{eq:complex}
\end{equation}
with $v(x,0)=u^{(1)}(x,0)$. It can be verified directly by It\^{o}'s
formula that 
\[
v(x,t)= {1 \over \sqrt{1+(2+\mu^{2})t}} 
\exp \biggl\{ {-\frac{(x+\sqrt{-1}\mu\BM_{t})^{2}}{2[1+(2+\mu^{2})t]}} \biggr\}
\]
solves \eqref{eq:complex} with the given initial condition. So we
can compute 
\begin{align}
\E|\bm{u}(x,t)|^{p} & =\E|v(x,t)|^{p}\label{eq:example-lp}\\
 & =\frac{1}{\sqrt{2\pi t}}\frac{1}{[1+(2+\mu^{2})t]^{p/2}}\,\me^{-\frac{px^{2}}{2[1+(2+\mu^{2})t]}}\int_{\R}\me^{-\frac{y^{2}}{2t}\bigl[1-\frac{p\mu^{2}t}{1+(2+\mu^{2})t}\bigr]}\vd y.\nonumber 
\end{align}
It is noticed that
\[
1-\frac{p\mu^{2}t}{1+(2+\mu^{2})t}\to\frac{2-(p-1)\mu^{2}}{2+\mu^{2}}\quad\text{as}\;t\to\infty,
\]
which implies that if
\begin{equation}
\mu^{2}>\frac{2}{p-1},\label{eq:exam-cond}
\end{equation}
the integral in \eqref{eq:example-lp} diverges for large $t$, and
$\E|\bm{u}(x,t)|^{p}=\infty$ for every $x$. 
\end{example}
A major contribution of this paper is the finding of a general coercivity condition that ensures us to construct a general $C^{2+\hold}$-theory
for system \eqref{eq:main}.
The basic idea is to impose an appropriate correction term involving
$p$ to the left-hand side of \eqref{eq:parabolic}. More specifically,
we introduce
\begin{defn}[MSP condition]
\label{cond:msp}Let $p\in[2,\infty)$. The coefficients $a=(a_{\alpha\beta}^{ij})$ and $\sigma=(\sigma_{\alpha\beta}^{ik})$ are said to satisfy the
{\em modified stochastic parabolicity (MSP) condition} if there are measurable functions
$\lambda_{\alpha\beta}^{ik}:\R^{d}\times[0,\infty)\times\PS\to\R$ with $\lambda_{\alpha\beta}^{ik}=\lambda_{\beta\alpha}^{ik}$,
such that 
\begin{equation}
\mathcal{A}_{\alpha\beta}^{ij}(p,\lambda):=2a_{\alpha\beta}^{ij}-\sigma_{\gamma\alpha}^{ik}\sigma_{\gamma\beta}^{jk}-(p-2)(\sigma_{\gamma\alpha}^{ik}-\lambda_{\gamma\alpha}^{ik})(\sigma_{\gamma\beta}^{jk}-\lambda_{\gamma\beta}^{jk})\label{eq:Aij}
\end{equation}
satisfy the \emph{Legendre\textendash Hadamard condition}: there is
a constant $\kappa>0$ such that
\begin{equation}
\label{eq:lp-condition}
\mathcal{A}_{\alpha\beta}^{ij}(p,\lambda)\,\xi_{i}\xi_{j}\eta^{\alpha}\eta^{\beta}\ge\kappa|\xi|^{2}|\eta|^{2}
\quad\forall\,\xi\in\R^{d},\ \eta\in\R^{N}
\end{equation}
everywhere on $\R^{d}\times[0,\infty)\times\PS$.
\end{defn}

In particular, the following criteria for the MSP condition, simplified by taking $\lambda_{\alpha\beta}^{ik}=0$ and $\lambda_{\alpha\beta}^{ik}=(\sigma_{\alpha\beta}^{ik}+\sigma_{\beta\alpha}^{ik})/2$ 
respectively in \eqref{eq:Aij}, could be
very convenient in applications. 

\begin{lem}\label{lem1}
The MSP condition is satisfied if either 
\begin{enumerate}
\item[{\em (i)}] $2a_{\alpha\beta}^{ij}-(p-1)\sigma_{\gamma\alpha}^{ik}\sigma_{\gamma\beta}^{jk}$
or
\item[{\em (ii)}] $2a_{\alpha\beta}^{ij}-\sigma_{\gamma\alpha}^{ik}\sigma_{\gamma\beta}^{jk}-(p-2)\widehat{\sigma}_{\gamma\alpha}^{ik}\widehat{\sigma}_{\gamma\beta}^{jk}$
with $\widehat{\sigma}_{\alpha\beta}^{ik}:=(\sigma_{\alpha\beta}^{ik}-\sigma_{\beta\alpha}^{ik})/2$
\end{enumerate}
satisfies the Legendre\textendash Hadamard condition. 
\end{lem}

Evidently, the MSP condition is \emph{invariant} under change of basis of $\R^{d}$ or under orthogonal transformation of $\R^{N}$. 
Also the Legendre\textendash Hadamard
condition (see for example \cite{giaquinta1993introduction}) 
is more general than the strong ellipticity condition. The MSP condition
coincides with the Legendre\textendash Hadamard condition for PDE
systems and the stochastic parabolicity condition for SPDEs. Besides when
$p=2$ it becomes
\begin{equation}
\big(2a_{\alpha\beta}^{ij}-\sigma_{\gamma\alpha}^{ik}\sigma_{\gamma\beta}^{jk}\big)\xi_{i}\xi_{j}\eta^{\alpha}\eta^{\beta}\ge\kappa|\xi|^{2}|\eta|^{2}\quad\forall\,\xi\in\R^{d},\ \eta\in\R^{N}\label{eq:p2}
\end{equation}
which is weaker than \eqref{eq:parabolic} used in \cite{Kim2013w}.
Moreover, the case (ii) in Lemma~\ref{lem1} shows that the MSP condition is also reduced
to \eqref{eq:p2} if the matrices $B^{ik}=[\sigma_{\alpha\beta}^{ik}]_{N\times N}$
are close to be symmetric. 
Nevertheless, the generality of the MSP condition cannot be covered by these cases in Lemma~\ref{lem1}, which is illustrated by Example~\ref{lem:last} in the final section.

Example \ref{exa:counterexam} illustrates that  in \eqref{eq:Aij} the coefficient of the correction term $p-2$ is \emph{optimal} to guarantee the Schauder regularity for the SPDEs \eqref{eq:main}. Indeed, if $p>2$ is fixed and
the coefficient $p-2$ in \eqref{eq:Aij} drops down a bit to $p-2-\eps>0$,
we can choose the value of $\mu$ satisfying
\[
\frac{2}{p-1}<\mu^{2}<\frac{2}{p-1-\eps},
\]
then it is easily verified that system \eqref{eq:ceg} satisfies \eqref{eq:lp-condition}
in this setting by taking $\lambda_{\alpha\beta}^{ik}=0$ and $p-2$ replaced by $p-2-\eps$. However,
Example \ref{exa:counterexam} has showed that when $t$ is large
enough $\E|\bm{u}(x,t)|^{p}$ becomes infinite for such a choice of
$\mu$,  let alone the $\cfun_{p}^{\hold}$-norm of the solution. More
examples in this respect are discussed in the final section.

Technically speaking, the MSP condition is explicitly used
to derive a class of mixed norm estimates for the model system in
the space $L^{p}(\PS;W_{2}^{n})$. A similar issue was addressed in
\cite{brzezniak2012stochastic} for a nonlocal SPDE. 
Owing to Sobolev's embedding the mixed norm estimates lead to
the local boundedness of $\E|\pd^{m}\bm{u}(x,t)|^{p}$, which plays a key
role in the derivation of the foundamental
interior estimate of Schauder-type for system \eqref{eq:main}.


The paper is organised as follows. In the next section we introduce
some notations and state our main results. In Sections \ref{sec:Integral-estimates-for}
and \ref{sec:Interior-H=0000F6lder-estimates} we consider the model
system
\[
\md u^{\alpha}=\big(a_{\alpha\beta}^{ij}\partial_{ij}u^{\beta}+f_{\alpha}\big)\md t+\big(\sigma_{\alpha\beta}^{ik}\partial_{i}u^{\beta}+g_{\alpha}^{k}\big)\vd\BM_{t}^{k},
\]
where the coefficients $a$ and $\sigma$ are \emph{random} but independent
of $x$. We prove the crucial mixed norm estimates in Section \ref{sec:Integral-estimates-for},
and then establish the interior H\"{o}lder estimate in Section \ref{sec:Interior-H=0000F6lder-estimates}.
In Section \ref{sec:H=0000F6lder-estimates-for} we complete the proofs
of our main results. 
The final
section is devoted to more comments and examples on the sharpness
and flexibility of the MSP condition.

\section{\label{sec:Main-results}Main results}

Let us first introduce our working spaces and associated notation.
A Banach-space valued H\"{o}lder continuous function is defined analogously
to the classical H\"{o}lder continuous function. Let $E$ be a 
Banach space, $\Dom$ a domain in $\R^{d}$ and $I$ an interval.
We define the parabolic modulus 
\[
|X|_{{\rm p}}=|x|+\sqrt{|t|}\quad\text{for}\ X=(x,t)\in Q:=\Dom\times I.
\]
For a space-time function $\bm{u}:Q\to E$, we define
\begin{align*}
[\bm{u}]_{m;Q}^{E} & :=\sup \{\|\pd^{\mathfrak{s}}\bm{u}(X)\|_{E}: X=(x,t)\in Q,\, |\mathfrak{s}| = m\},\\
|\bm{u}|_{m;Q}^{E} & :=\max\{[\bm{u}]_{k;Q}^{E}: k\le m \},\\
[\bm{u}]_{m+\hold;Q}^{E} & :=\sup_{|\mathfrak{s}|=m}\sup_{t\in I}\sup_{x,y\in\Dom}\frac{\|\pd^{\mathfrak{s}}\bm{u}(x,t)-\pd^{\mathfrak{s}}\bm{u}(y,t)\|_{E}}{|x-y|^{\hold}},\\
|\bm{u}|_{m+\hold;Q}^{E} & :=|\bm{u}|_{m;Q}^{E}+[\bm{u}]_{m+\hold;Q}^{E},\\{}
[\bm{u}]_{(m+\hold,\hold/2);Q}^{E} & :=\sup_{|\mathfrak{s}|=m}\sup_{X,Y\in Q}\frac{\|\pd^{\mathfrak{s}}\bm{u}(X)-\pd^{\mathfrak{s}}\bm{u}(Y)\|_{E}}{|X-Y|_{{\rm p}}^{\hold}},\\
|\bm{u}|_{(m+\hold,\hold/2);Q}^{E} & :=|\bm{u}|_{m;Q}^{E}+[\bm{u}]_{(m+\hold,\hold/2);Q}^{E}
\end{align*}
with $m\in\mathbf{N}:=\{0,1,2,\dots\}$ and $\hold\in(0,1)$, where
$\mathfrak{s}=(\mathfrak{s}_1,\cdots,\mathfrak{s}_d)\in\mathbf{N}^{d}$ with $|\mathfrak{s}|=\sum_{i=1}^{d}\mathfrak{s}_i$,
and all the derivatives of an $E$-valued function are defined with respect
to the \emph{spatial variable} in the strong sense, see \cite{hille1957functional}. In the following context, the space $E$ is either i) an Euclidean space, ii) the space $\ell^2$, 
or iii) $L_{\om}^{p}:=L^{p}(\PS)$ (abbreviation for  $L^p_\om$ for both $L^{p}(\PS;\R^N)$ or $L^{p}(\PS;\ell^2)$). 
We omit the superscript in cases (i) and (ii), and in case (iii), we
introduce some new notation:
\begin{align*}
\lbr\bm{u}\rbr_{m+\hold,p;Q} & :=[\bm{u}]_{m+\hold;Q}^{L_{\om}^{p}}\ ,\quad\lbr\bm{u}\rbr_{(m+\hold,\hold/2),p;Q}:=[\bm{u}]_{(m+\hold,\hold/2);Q}^{L_{\om}^{p}},\\
\bbr\bm{u}\bbr_{m+\hold,p;Q} & :=|\bm{u}|_{m+\hold;Q}^{L_{\om}^{p}}\ ,\quad\!\!\bbr\bm{u}\bbr_{(m+\hold,\hold/2),p;Q}:=|\bm{u}|_{(m+\hold,\hold/2);Q}^{L_{\om}^{p}}.
\end{align*}
As the random fields in this paper take values in different spaces like $\R^N$ (say, ${\bm u}$ and ${\bm f}$) or $\ell^2$ (say, ${\bm g}$), 
we shall use $|\cdot|$ uniformly for the standard norms in Euclidean spaces and in $\ell^2$, 
and $L^p_\om$ for both $L^{p}(\PS;\R^N)$ and $L^{p}(\PS;\ell^2)$; 
the specific meaning of the notation can be easily understood in context.

\begin{defn*}
The H\"{o}lder classes $C_{x}^{m+\hold}(Q;L_{\om}^{p})$ and $C_{x,t}^{m+\hold,\hold/2}(Q;L_{\om}^{p})$
are defined as the sets of all \emph{predictable} random fields $\bm{u}$
defined on $Q\times\PS$ and taking values in an Euclidean space or $\ell^2$ such
that $\bbr\bm{u}\bbr_{m+\hold,p;Q}$ and $\bbr\bm{u}\bbr_{(m+\hold,\hold/2),p;Q}$
are finite, respectively. 
\end{defn*}
The following notation for special domains are frequently used:
\[
B_{r}(x)=\big\{ y\in\R^{d}:|y-x|<r\big\},\quad Q_{r}(x,t)=B_{r}(x)\times(t-r^{2},t],
\]
and $B_{r}=B_{r}(0)$, $Q_{r}=Q_{r}(0,0)$, and also 
\[
\mathcal{Q}_{r,T}(x):=B_{r}(x)\times(0,T],\quad\mathcal{Q}_{r,T}=\mathcal{Q}_{r,T}(0)\quad\text{and}\quad\mathcal{Q}_{T}:=\R^{d}\times(0,T].
\]

\begin{assumption*}
The following conditions are used throughout the paper unless otherwise
stated:
\begin{enumerate}
\item[{\em i)}] For all $i,j=1,\dots,d$ and $\alpha,\beta=1,\dots,N$, the random
fields $a_{\alpha\beta}^{ij},b_{\alpha\beta}^{i},c_{\alpha\beta}$
and $f_{\alpha}$ are real-valued, and $\sigma_{\alpha\beta}^{i},\nu_{\alpha\beta}$
and $g_{\alpha}$ are $\ell^{2}$-valued; all of them are predictable. 
\item[{\em ii)}] $a_{\alpha\beta}^{ij}$ and $\sigma_{\alpha\beta}^{i}$ satisfy the MSP condition with some $p\in[2,\infty)$.
\item[{\em iii)}] For some $\hold\in(0,1)$, the classical $C_{x}^{\hold}$-norms of
$a_{\alpha\beta}^{ij},b_{\alpha\beta}^{i}$ and $c_{\alpha\beta}$,
and the $C_{x}^{1+\hold}$-norms of $\sigma_{\alpha\beta}^{i}$ and
$\nu_{\alpha\beta}$ are all dominated by a constant $K$.
\end{enumerate}
\end{assumption*}
We are ready to state the main results of the paper. The first result
is the a priori \emph{interior H\"{o}lder estimates} for system \eqref{eq:main}.
\begin{thm}
\label{thm:interior}Under the above setting, there exist two constants
$\rho_{0}\in(0,1)$ and $C>0$, both depending only on $d,N,\kappa,K,p$
and $\hold$, such that if $\bm{u}\in C^{2+\hold}_x(Q_1(X);L^p_\omega)$
satisfies \eqref{eq:main} in $Q_1(X)$ with $X=(x,t)\in\R^d\times[1,\infty)$, then
\begin{align}
\label{eq:interior}
&\rho^{2+\hold}\lbr\pd^{2}\bm{u}\rbr_{(\hold,\hold/2),p;Q_{\rho/2}(X)} \\
& \le C\Big\{\rho^{2}\bbr\bm{f}\bbr_{\hold,p;Q_{\rho}(X)}+\rho\,\bbr\bm{g}\bbr_{1+\hold,p;Q_{\rho}(X)}+\rho^{-\frac{d}{2}}\big[\E\|\bm{u}\|_{L^{2}(Q_{\rho}(X))}^{p}\big]^{\!\frac{1}{p}}\Big\}
\nonumber
\end{align}
for any $\rho\in(0,\rho_{0}]$, provided the right-hand side is finite.
\end{thm}

By rescaling one can obtain the local estimate arround any point $X\in \R^d\times(0,\infty)$.
\smallskip

The second theorem is regarding the global H\"{o}lder estimate and solvability
for the Cauchy problem for system \eqref{eq:main} with zero initial
condition.
\begin{thm}
\label{thm:cauchy}Under the above setting, if $\bm{f}\in C^{\hold}_x(\mathcal{Q}_{T};L_{\om}^{p})$
and $\bm{g}\in C_{x}^{1+\hold}(\mathcal{Q}_{T};L_{\om}^{p})$ with
$T>0$, then system \eqref{eq:main} with the initial condition
\[
\bm{u}(x,0)=\bm{0}\quad\forall\,x\in\R^{d}
\]
admits a unique solution $\bm{u}\in C_{x,t}^{2+\hold,\hold/2}(\mathcal{Q}_{T};L_{\om}^{p})$,
and it satisfies the estimate
\begin{equation}
\bbr\bm{u}\bbr_{(2+\hold,\hold/2),p;\mathcal{Q}_{T}}\le C\,\me^{CT}\big(\bbr\bm{f}\bbr_{\hold,p;\mathcal{Q}_{T}}+\bbr\bm{g}\bbr_{1+\hold,p;\mathcal{Q}_{T}}\big),\label{eq:global}
\end{equation}
where the constant $C$ depends only on $d,N,\kappa,K,p$ and $\hold$.
\end{thm}
\begin{rem*}
Theorem \ref{thm:cauchy} still holds true if the system is considered
on the torus $\mathbf{T}^{d}=\R^{d}/\mathbf{Z}^{d}$ instead of $\R^{d}$.
\end{rem*}
\begin{rem*}
The above theorems show that the solutions possess the H\"older continuity in time even with time-irregular coefficients and free terms.
A similar property of classical PDEs is well-known in the literature, see for example~\cite{lieberman1996second,dong2015schauder} and references therein.
In view of an anisotropic Kolmogorov
continuity theorem (see \cite{dalang2007hitting}) the solution obtained in Theorem~\ref{thm:cauchy} has a modification that is H\"older continuous jointly in space and time.
\end{rem*}

\section{\label{sec:Integral-estimates-for}Integral estimates for the model
system}

Throughout this section we assume that $a_{\al\be}^{ij}$ and $\si_{\al\be}^{ik}$
depend only on $(t,\om)$, but \emph{independent of $x$}, satisfying
the MSP condition
(in this case {\em $\lambda_{\al\be}^{ik}$ is chosen to be independent of $x$}) and 
\begin{equation}
|a_{\alpha\beta}^{ij}|,\;|\sigma_{\alpha\beta}^{ij}|\leq K,\quad\forall t,\om,\label{eq:modeq-coe-bound}
\end{equation}
and we consider the following model system
\begin{equation}
\md u^{\alpha}=\big(a_{\alpha\beta}^{ij}\partial_{ij}u^{\beta}+f_{\alpha}\big)\md t+\big(\sigma_{\alpha\beta}^{ik}\partial_{i}u^{\beta}+g_{\alpha}^{k}\big)\vd\BM_{t}^{k}.\label{eq:model}
\end{equation}
The aim of this section is to derive several auxiliary estimates for
the model system which are used to prove the interior H\"{o}lder estimate
in the next section.

In this section and the next, {\em we may consider \eqref{eq:model} in the entire space $\R^n\times\R$.}
On the one hand, we can always extend \eqref{eq:main} and \eqref{eq:model} to the entire space if we require $u(x,0)=0$.
Indeed, the zero extensions of $\bm{u}$, $\bm{f}$ and $\bm{g}$
(i.e., these functions are defined to be zero for $t<0$) satisfy the equations in the entire space, 
where the extension of coefficients and Wiener processes are quite easy; for example,
we can define $a_{\alpha\beta}^{ij}(t) = \delta^{ij}$ and $\sigma_{\alpha\beta}^{ik}=0$ for $t<0$, and $\BM_t := \tilde{\BM}_{-t}$ for $t<0$ with $\tilde{\BM}$ being an independent copy of $\BM$.
On the other hand,  we mainly concern the local estimates for the equation \eqref{eq:model} in the following two sections, so we can only focus on the estimates around the origin on account of a translation.
Indeed, we can reduce the estimates around a point $(x_0,t_0)$ to the estimates around the origin by use of the change of variables $(x,t)\mapsto (x-x_0,t-t_0)$.

Let $\bO\in\R^{d}$ and $H^{m}(\bO)=W_{2}^{m}(\bO)$ be the usual
Sobolev spaces. Let $I\subset\R$ be an interval and $Q=\bO\times I$. For $p,q\in[1,\infty]$,
define 
\[
L_{\om}^{p}L_{t}^{q}H_{x}^{m}(Q):=L^{p}(\PS;L^{q}(I;H^{m}(\bO;\R^{N}))).
\]
In what follows, we denote $\partial^m \bm{u}$ the set of all $m$-order derivatives of a function $\bm{u}$.
These $\partial^m \bm{u}(x)$ for each $x$ and $(\omega,t)$ are regarded as elements
of a Euclidean space of proper dimension.

Our $\mathcal{C}^{2+\hold}$-theory is grounded in the following mixed
norm estimates for model system \eqref{eq:model}, in which the modified
stochastic parabolicity condition \eqref{eq:lp-condition} plays a
key role. 
\begin{thm}
\label{thm:globe-lpl2}Let $p\in[2,\infty)$ and $m\geq 0$. Suppose
$\bm{f}\in L_{\om}^{p}L_{t}^{2}H_{x}^{m-1}(\mathcal{Q}_{T})$ and
$\bm{g}\in L_{\om}^{p}L_{t}^{2}H_{x}^{m}(\mathcal{Q}_{T})$.
Then \eqref{eq:model} with zero initial value admits a unique solution
$\bm{u}\in L_{\om}^{p}L_{t}^{\infty}H_{x}^{m}(\mathcal{Q}_{T})\cap L_{\om}^{p}L_{t}^{2}H_{x}^{m+1}(\mathcal{Q}_{T})$.
Moreover, for any multi-index $\mathfrak{s}$ such that $|\mathfrak{s}|\leq m$,
\begin{equation}
\|\pd^{\mathfrak{s}}\bm{u}\|_{L_{\om}^{p}L_{t}^{\infty}L_{x}^{2}}+\|\pd^{\mathfrak{s}}\bm{u}_{x}\|_{L_{\om}^{p}L_{t}^{2}L_{x}^{2}}\leq C\Big(\|\pd^{\mathfrak{s}}\bm{f}\|_{L_{\om}^{p}L_{t}^{2}H_{x}^{-1}}+\|\pd^{\mathfrak{s}}\bm{g}\|_{L_{\om}^{p}L_{t}^{2}L_{x}^{2}}\Big),\label{eq:LPL2-est}
\end{equation}
where the constant $C$ depends only on $d,p,T,\kappa,$ and $K$. 
\end{thm}
The proof of Theorem \ref{thm:globe-lpl2} is postponed to the end
of this section. A quick consequence of this theorem is the following
local estimates for model equations with smooth free terms. 
\begin{prop}
\label{lem:local-LPL2-1}Let $m\geq1$, $p\geq2$, $r>0$ and $0<\theta<1$,
and let ${\bm{u}}\in L_{\om}^{p}L_{t}^{\infty}H_{x}^{m}(Q_{r})\cap L_{\om}^{p}L_{t}^{2}H_{x}^{m+1}(Q_{r})$
solve \eqref{eq:model} in $Q_{r}$ with ${\bm{f}}\in L_{\om}^{p}L_{t}^{2}H_{x}^{m-1}(Q_{r})$
and ${\bm{g}}\in L_{\om}^{p}L_{t}^{2}H_{x}^{m}(Q_{r})$. Then there
is a constant $C=C(d,p,\kappa,K,m,\theta)$ such that 
\begin{gather}
\|\partial^{m}{\bm{u}}\|_{L_{\om}^{p}L_{t}^{\infty}L_{x}^{2}(Q_{\theta r})}+\|\partial^{m}{\bm{u}}_{x}\|_{L_{\om}^{p}L_{t}^{2}L_{x}^{2}(Q_{\theta r})}\le Cr^{-m-1}\|{\bm{u}}\|_{L_{\om}^{p}L_{t}^{2}L_{x}^{2}(Q_{r})}\label{eq:locallpl2-est-1-1}\\
+ C \sum_{k=0}^{m-1} r^{-m+k+1} \|\partial^{k} \bm{f}\|_{L^p_\om L^2_t L^2_{x}(Q_{r})} 
+ C \sum_{k=0}^{m} r^{-m+k} \|\partial^k \bm{g}\|_{L^p_\om L^2_t L^2_{x}(Q_{r})}.\nonumber 
\end{gather}
Consequently, for $2(m-|\mathfrak{s}|)>d$, 
\begin{gather}
\|\sup_{Q_{\theta r}}|\partial^{\mathfrak{s}}{\bm{u}}|\|_{L_{\om}^{p}}\le Cr^{-|\mathfrak{s}|-d/2-1}\|{\bm{u}}\|_{L_{\om}^{p}L_{t}^{2}L_{x}^{2}(Q_{r})}
\label{eq:locallpl2-est-2-1}\\
+ C \sum_{k=0}^{m-1} r^{-|\mathfrak{s}|-d/2+k+1} \|\partial^{k} \bm{f}\|_{L^p_\om L^2_t L^2_{x}(Q_{r})} 
+ C \sum_{k=0}^{m} r^{-|\mathfrak{s}|-d/2+k} \|\partial^k \bm{g}\|_{L^p_\om L^2_t L^2_{x}(Q_{r})}.\nonumber 
\end{gather}
\end{prop}
\begin{proof}
It suffices to prove (\ref{eq:locallpl2-est-1-1}) as (\ref{eq:locallpl2-est-2-1})
follows from (\ref{eq:locallpl2-est-1-1}) immediately by Sobolev's
embedding theorem \cite[Theorem~4.12]{adams2003sobolev}. 
Moreover for general $r>0$, we can apply the obtained estimates for $r=1$ to the rescaled function
\[
{\bm{v}}(x,t):={\bm{u}}(rx,r^{2}t),\quad\forall(x,t)\in\R^{d}\times\R
\]
which solves the equation 
\begin{equation}
\md v^{\alpha}(x,t)=\big(a_{\alpha\beta}^{ij}(r^{2}t)\partial_{ij}v^{\beta}(x,t)+F_{\alpha}\big)\md t+\big(\sigma_{\alpha\beta}^{ik}(r^{2}t)\partial_{i}v^{\beta}(x,t)+G_{\alpha}^{k}\big)\vd\beta_{t}^{k},\label{eq:model-1-1}
\end{equation}
with 
\[
F_{\alpha}(x,t)=r^{2}f_{\alpha}(rx,r^{2}t),\quad G_{\alpha}^{k}(x,t)=rg(rx,r^{2}t),\quad\beta_{t}^{k}=r^{-1}\BM_{r^{2}t}^{k}.
\]
Obviously, $\beta^{k}$ are mutually independent standard Wiener processes.

For any $\theta\in(0,1)$, choose cut-off functions $\zeta^{\ell}\in C_{0}^{\infty}(\R^{d+1})$,
$\ell=1,2$, satisfying i) $0\leq\zeta^{\ell}\leq1$, ii) $\zeta^{1}=1$
in $Q_{\sqrt{\theta}}$ and $\zeta^{1}=0$ outside $Q_{1}$, and iii)
$\zeta^{2}=1$ in $Q_{\theta}$ and $\zeta^{2}=0$ outside $Q_{\sqrt{\theta}}$.
Then ${\bm{v}}_{\ell}=\zeta^{\ell}{\bm{u}}$ ($\ell=1,2$) satisfy
\begin{equation}
\md v_{l}^{\alpha}=\big(a_{\alpha\beta}^{ij}\partial_{ij}v_{\ell}^{\beta}+\tilde{f}_{\ell,\alpha}\big)\md t+\big(\sigma_{\alpha\beta}^{ik}\partial_{i}v_{\ell}^{\beta}+\tilde{g}_{\ell,\alpha}^{k}\big)\vd\BM_{t}^{k},\quad\ell=1,2,\label{eq:Loc-lpl2-1-1}
\end{equation}
where 
\begin{align*}
\tilde{f}_{\ell,\alpha} & =\zeta^{\ell}f_{\al}-a_{\alpha\beta}^{ij}(\zeta_{x_{i}}^{\ell}u^{\beta})_{x_{j}}+a_{\alpha\beta}^{ij}\zeta_{x_{i}x_{j}}^{\ell}u^{\beta}+(\pd_{t}\zeta^{\ell})u^{\alpha},\\
\tilde{g}_{\ell,\alpha}^{k} & =\zeta^{\ell}g_{\alpha}^{k}-\sigma_{\alpha\beta}^{ik}\zeta_{x_{i}}^{\ell}u^{\beta},\quad\ell=1,2.
\end{align*}
Applying Theorem \ref{thm:globe-lpl2} to (\ref{eq:Loc-lpl2-1-1})
for $\ell=1$, $\mathfrak{s}=0$ and for $\ell=2$, $|\mathfrak{s}|=1$,
we have 
\begin{align*}
 & \|{\bm{u}}\|_{L_{\om}^{p}L_{t}^{\infty}L_{x}^{2}(Q_{\sqrt{\theta}})}+\|{\bm{u}}_{x}\|_{L_{\om}^{p}L_{t}^{2}L_{x}^{2}(Q_{\sqrt{\theta}})}\\
 & \,\,\leq C\Big(\|{\bm{u}}\|_{L_{\om}^{p}L_{t}^{2}L_{x}^{2}(Q_{1})}+\|{\bm{f}}\|_{L_{\om}^{p}L_{t}^{2}L_{x}^{2}(Q_{1})}+\|{\bm{g}}\|_{L_{\om}^{p}L_{t}^{2}L_{x}^{2}(Q_{1})}\Big);\\
 & \|\pd^{\mathfrak{s}}{\bm{u}}\|_{L_{\om}^{p}L_{t}^{\infty}L_{x}^{2}(Q_{\theta})}+\|\pd^{\mathfrak{s}}{\bm{u}}_{x}\|_{L_{\om}^{p}L_{t}^{2}L_{x}^{2}(Q_{\theta})}\\
 & \,\,\leq C\Big(\|{\bm{u}}\|_{L_{\om}^{p}L_{t}^{2}L_{x}^{2}(Q_{\sqrt{\theta}})}+\|\pd^{\mathfrak{s}}{\bm{u}}\|_{L_{\om}^{p}L_{t}^{2}L_{x}^{2}(Q_{\sqrt{\theta}})}\|+\|{\bm{f}}\|_{L_{\om}^{p}L_{t}^{2}L_{x}^{2}(Q_{\sqrt{\theta}})}+\|\pd^{\mathfrak{s}}{\bm{g}}\|_{L_{\om}^{p}L_{t}^{2}L_{x}^{2}(Q_{\sqrt{\theta}})}\Big).
\end{align*}
Combining these two estimates, we have \eqref{eq:locallpl2-est-1-1}
for $m=1$. Higher order estimates follows from induction. The proof
is complete. 
\end{proof}
Another consequence of Theorem \ref{thm:globe-lpl2} is the following
lemma concerning the estimates for equation \eqref{eq:model} with
the Cauchy\textendash Dirichlet boundary conditions:
\begin{equation}
\bigg\{\begin{aligned}\bm{u}(x,0) & =0,\quad\forall\,x\in B_{r};\\
\bm{u}(x,t) & =0,\quad\forall\,(x,t)\in\partial B_{r}\times(0,T].
\end{aligned}
\label{eq:dirich}
\end{equation}

\begin{prop}
\label{lem:dirich}Let $\bm{f}=\bm{f}^{0}+\pd_{i}\bm{f}^{i}$ and
$\bm{f}^{0},\bm{f}^{1},\dots,\bm{f}^{d},\bm{g}\in L^p_\om L^2_t H^m_{x}(\mathcal{Q}_{r,r^2})$ for all $m\ge 0$.
Then problem \eqref{eq:model} and \eqref{eq:dirich} admits a unique
solution $\bm{u}\in L_{\om}^{2}L_{t}^{2}H_{x}^{1}(\mathcal{Q}_{r,r^2})$,
and for each $t\in(0,r^2)$, $\bm{u}(\cdot,t)\in L^{p}(\PS;C^{m}(B_{\eps};\R^{N}))$
with any $m\ge0$ and $\eps\in(0,r)$. Moreover, there is a constant
$C=C(n,p)$ such that
\begin{equation}
\|\bm{u}\|_{L_{\om}^{p}L_{t}^{2}L_{x}^{2}(\mathcal{Q}_{r,r^{2}})}\leq C\Big(r^{2}\|\bm{f}^{0}\|_{L_{\om}^{p}L_{t}^{2}L_{x}^{2}(\mathcal{Q}_{r,r^{2}})}+r\|(\bm{f}^{1},\dots,\bm{f}^{d},\bm{g})\|_{L_{\om}^{p}L_{t}^{2}L_{x}^{2}(\mathcal{Q}_{r,r^{2}})}\Big).\label{eq:dirich-est}
\end{equation}
\end{prop}
\begin{proof}
The existence, uniqueness and smoothness of the solution of problem
\eqref{eq:model} and \eqref{eq:dirich} follow from \cite[Theorem 4.8]{Kim2013w},
and \eqref{eq:dirich-est} from \eqref{eq:LPL2-est} and rescaling.
We remark that, although the results in \cite{Kim2013w} used condition
\eqref{eq:parabolic}, Lemma \ref{lem:LH} below ensures that those
results remain valid for the model equation \eqref{eq:model} under
condition \eqref{eq:p2} that is implied by the MSP condition.
\end{proof}
The following lemma is standard (cf. \cite{giaquinta1993introduction}).
\begin{lem}
\label{lem:LH}If the real numbers $A_{\alpha\beta}^{ij}$ satisfy
the Legendre\textendash Hadamard condition, then there exists a constant
$\epsilon>0$ depending only on $d,N$ and $\kappa$ such that
\[
\int_{\R^{d}}A_{\alpha\beta}^{ij}\pd_{i}u^{\alpha}\pd_{j}u^{\beta}\ge\epsilon\int_{\R^{d}}|\pd\bm{u}|^{2}
\]
for any $\bm{u}\in H^{1}(\R^{d};\R^{N})$. 
\end{lem}
The rest of this section is devoted to the proof of Theorem \ref{thm:globe-lpl2}.
\begin{proof}[Proof of Theorem \ref{thm:globe-lpl2}]
According to Theorem 2.3 in \cite{Kim2013w} the model system (\ref{eq:model})
with zero initial value admits a unique solution 
\[
{\bm{u}}\in L_{\om}^{2}L_{t}^{\infty}H_{x}^{m}(\mathcal{Q}_{T})\cap L_{\om}^{2}L_{t}^{2}H_{x}^{m+1}(\mathcal{Q}_{T}).
\]
Noting that $\bm{u}\in L_{\om}^{p}L_{t}^{\infty}H_{x}^{m}(\mathcal{Q}_{T})\cap L_{\om}^{p}L_{t}^{2}H_{x}^{m+1}(\mathcal{Q}_{T})$ follows from estimate (\ref{eq:LPL2-est}) by approximation,
 it remains to prove (\ref{eq:LPL2-est}). As we can differentiate
(\ref{eq:model}) with order $\mathfrak{s}$, it suffices to show
(\ref{eq:LPL2-est}) for $m=0$.

By It\^{o}'s formula, we derive
\begin{align}
\md|\bm{u}|^{2} & =\big[-\big(2a_{\al\be}^{ij}-\sigma_{\gamma\alpha}^{ik}\sigma_{\gamma\beta}^{jk}\big)\pd_{i}u^{\al}\pd_{j}u^{\be}+2a_{\al\be}^{ij}\pd_{i}\big(u^{\al}\pd_{j}u^{\be}\big)\big]\vd t\\
 & \quad+\big(2u^{\al}f_{\al}+2\si_{\al\be}^{ik}\pd_{i}u^{\be}g_{\al}^{k}+|\bm{g}|^{2}\big)\vd t+2\big(\si_{\al\be}^{ik}u^{\al}\pd_{i}u^{\be}+u^{\al}g_{\al}^{k}\big)\vd\BM_{t}^{k}.\nonumber 
\end{align}
Integrating with respect to $x$ over $\R^{d}$ and using the divergence
theorem, we have 
\begin{align}
 & \md\|\bm{u}(\cdot,t)\|_{L_{x}^{2}}^{2}\label{eq:l2-ito-2}\\
 & =\int_{\R^{d}}\big[-\big(2a_{\al\be}^{ij}-\sigma_{\gamma\alpha}^{ik}\sigma_{\gamma\beta}^{jk}\big)\pd_{i}u^{\al}\pd_{j}u^{\be}+2u^{\al}f_{\al}+2\si_{\al\be}^{ik}\pd_{i}u^{\be}g_{\al}^{k}+|\bm{g}|^{2}\big]\vd x\md t\nonumber \\
 & \quad+\int_{\R^{d}}2\big(\sigma_{\alpha\be}^{ik}u^{\al}\pd_{i}u^{\be}+u^{\al}g_{\al}^{k}\big)\vd x\vd\BM_{t}^{k}.\nonumber 
\end{align}
Applying It\^{o}'s formula to $\|\bm{u}(\cdot,t)\|_{L_{x}^{2}}^{p}$ gives
\begin{align*}
 & \md\|\bm{u}(\cdot,t)\|_{L_{x}^{2}}^{p}\\
 & =\frac{p}{2}\|\bm{u}\|_{L_{x}^{2}}^{p-2}\int_{\R^{d}}\big[-\big(2a_{\al\be}^{ij}-\sigma_{\gamma\alpha}^{ik}\sigma_{\gamma\beta}^{jk}\big)\pd_{i}u^{\al}\pd_{j}u^{\be}+2u^{\al}f_{\al}+2\si_{\al\be}^{ik}\pd_{i}u^{\be}g_{\al}^{k}+|\bm{g}|^{2}\big]\vd x\vd t\\
 & \quad+\frac{p(p-2)}{2}\bm{1}_{\{\|\bm{u}\|_{L_{x}^{2}}\neq0\}}\|\bm{u}\|_{L_{x}^{2}}^{p-4}\sum_{k}\text{\ensuremath{\bigg[}}\int_{\R^{d}}\big(\sigma_{\alpha\be}^{ik}u^{\al}\pd_{i}u^{\be}+u^{\al}g_{\al}^{k}\big)\vd x\bigg]^{2}\vd t\\
 & \quad+p\|\bm{u}\|_{L_{x}^{2}}^{p-2}\int_{\R^{d}}\big(\sigma_{\alpha\be}^{ik}u^{\al}\pd_{i}u^{\be}+u^{\al}g_{\al}^{k}\big)\vd x\vd\BM_{t}^{k}.
\end{align*}
Recalling the MSP condition for the definition of $\lambda_{\alpha\beta}^{ik}$
and that $\lambda_{\alpha\beta}^{ik}=\lambda_{\beta\alpha}^{ik}$,
we compute
\begin{align*}
\si_{\al\be}^{ik}u^{\al}\pd_{i}u^{\be} & =(\sigma_{\alpha\beta}^{ik}-\lambda_{\alpha\beta}^{ik})u^{\al}\pd_{i}u^{\be}+\lambda_{\al\be}^{ik}u^{\al}\pd_{i}u^{\be}\\
 & =(\sigma_{\alpha\beta}^{ik}-\lambda_{\alpha\beta}^{ik})u^{\al}\pd_{i}u^{\be}+\frac{1}{2}\lambda_{\al\be}^{ik}\pd_{i}(u^{\al}u^{\be}),
\end{align*}
so by the integration by parts,
\[
\int_{\R^{d}}\sigma_{\alpha\be}^{ik}u^{\al}\pd_{i}u^{\be}\vd x=\int_{\R^{d}}(\sigma_{\alpha\beta}^{ik}-\lambda_{\alpha\beta}^{ik})u^{\al}\pd_{i}u^{\be}\vd x.
\]
Using the MSP condition and Lemma \ref{lem:LH}, we can dominate
the highest order terms: 
\begin{align*}
 & -\|\bm{u}\|_{L_{x}^{2}}^{2}\int_{\R^{d}}\!\!\big(2a_{\al\be}^{ij}-\sigma_{\gamma\alpha}^{ik}\sigma_{\gamma\beta}^{jk}\big)\pd_{i}u^{\al}\pd_{j}u^{\be}\vd x+(p-2)\sum_{k}\bigg(\!\int_{\R^{d}}\!\!\sigma_{\alpha\be}^{ik}u^{\al}\pd_{i}u^{\be}\vd x\!\bigg)^{\!2}\\
\le & -\|\bm{u}\|_{L_{x}^{2}}^{2}\int_{\R^{d}}\big(2a_{\al\be}^{ij}-\sigma_{\gamma\alpha}^{ik}\sigma_{\gamma\beta}^{jk}\big)\pd_{i}u^{\al}\pd_{j}u^{\be}\vd x\\
 & \qquad+(p-2)\|\bm{u}\|_{L_{x}^{2}}^{2}\sum_{k,\gamma}\int_{\R^{d}}\big[(\sigma_{\gamma\beta}^{ik}-\lambda_{\gamma\beta}^{ik})\pd_{i}u^{\be}\big]^{2}\vd x\\
= & -\|\bm{u}\|_{L_{x}^{2}}^{2}\int_{\R^{d}}\bigl[2a_{\al\be}^{ij}-\sigma_{\gamma\alpha}^{ik}\sigma_{\gamma\beta}^{jk}-(p-2)(\sigma_{\gamma\alpha}^{ik}-\lambda_{\gamma\alpha}^{ik})(\sigma_{\gamma\beta}^{jk}-\lambda_{\gamma\beta}^{jk})\bigr]\pd_{i}u^{\al}\pd_{j}u^{\be}\vd x\\
\le & -\epsilon\|\bm{u}\|_{L_{x}^{2}}^{2}\|\pd\bm{u}\|_{L_{x}^{2}}^{2}.
\end{align*}
So we have 
\begin{align}
 & \md\|\bm{u}(\cdot,t)\|_{L_{x}^{2}}^{p}\label{eq:lp00}\\
 & \le\frac{p}{2}\|\bm{u}\|_{L_{x}^{2}}^{p-2}\bigl(-\epsilon\|\pd\bm{u}\|_{L_{x}^{2}}^{2}+2\|\bm{u}\|_{H_{x}^{1}}\|\bm{f}\|_{H_{x}^{-1}}+C\|\bm{g}\|_{L_{x}^{2}}^{2}+C\|\pd\bm{u}\|_{L_{x}^{2}}\|\bm{g}\|_{L_{x}^{2}}\bigr)\vd t\nonumber \\
 & \quad+p\|\bm{u}\|_{L_{x}^{2}}^{p-2}\int_{\R^{d}}\big(\text{\ensuremath{\sigma}}_{\al\be}^{ik}u^{\al}\pd_{i}u^{\be}+u^{\al}g_{\al}^{k}\big)\vd x\vd\BM_{t}^{k}\nonumber \\
 & \leq\Big[-\frac{p\epsilon}{4}\|\bm{u}\|_{L_{x}^{2}}^{p-2}\|\pd\bm{u}\|_{L_{x}^{2}}^{2}+C\|\bm{u}\|_{L_{x}^{2}}^{p}+C\|\bm{u}\|_{L_{x}^{2}}^{p-2}\big(\|\bm{f}\|_{H_{x}^{-1}}^{2}+\|\bm{g}\|_{L_{x}^{2}}^{2}\big)\Big]\vd t\nonumber \\
 & \quad+p\|\bm{u}\|_{L_{x}^{2}}^{p-2}\int_{\R^{d}}\big(\text{\ensuremath{\sigma}}_{\al\be}^{ik}u^{\al}\pd_{i}u^{\be}+u^{\al}g_{\al}^{k}\big)\vd x\vd\BM_{t}^{k}.\nonumber 
\end{align}
Integrating with respect to time on $[0,s]$ for any $s\in[0,T]$,
and keeping in mind the initial condition $\bm{u}(x,0)\equiv0$, we
know that 
\begin{align}
 & \|\bm{u}(s)\|_{L_{x}^{2}}^{p}+\frac{p\epsilon}{4}\int_{0}^{s}\|\bm{u}\|_{L_{x}^{2}}^{p-2}\|\pd\bm{u}\|_{L_{x}^{2}}^{2}\vd t\nonumber \\
 & \le C\int_{0}^{s}\Big[\|\bm{u}(t)\|_{L_{x}^{2}}^{p}+\|\bm{u}\|_{L_{x}^{2}}^{p-2}\big(\|\bm{f}\|_{H_{x}^{-1}}^{2}+\|\bm{g}\|_{L_{x}^{2}}^{2}\big)\Big]\vd t\label{eq:lpl2-1}\\
 & \quad+\int_{0}^{s}p\|\bm{u}\|_{L_{x}^{2}}^{p-2}\int_{\R^{d}}\big[\text{\ensuremath{\sigma}}_{\al\be}^{ik}u^{\al}\pd_{i}u^{\be}+u^{\al}g_{\al}^{k}\big]\vd x\vd\BM_{t}^{k},\quad\text{a.s.}\nonumber 
\end{align}
Let $\tau\in[0,T]$ be  a stopping time  such that 
\[
\E\sup_{t\in[0,\tau]}\|\bm{u}(t)\|_{L_{x}^{2}}^{p}+\E\biggl(\int_{0}^{\tau}\|\pd\bm{u}(t)\|_{L_{x}^{2}}^{2}\vd t\biggr)^{\frac{p}{2}}<\infty.
\]
Then it is easily verified that the last term on the right-hand side
of (\ref{eq:lpl2-1}) is a martingale with parameter $s$. Taking
the expectation on both sides of (\ref{eq:lpl2-1}), and by Young's
inequality and Gronwall's inequality, we can obtain that 
\begin{align}
\sup_{t\in[0,T]}\E\|{\bm{u}}(t\wedge\tau)\|_{L_{x}^{2}}^{p} & +\E\int_{0}^{\tau}\|{\bm{u}}(t)\|_{L_{x}^{2}}^{p-2}\|\pd{\bm{u}}(t)\|_{L_{x}^{2}}^{2}\vd t\label{eq:Lpl2-2}\\
 & \le C\E\int_{0}^{\tau}\|{\bm{u}}(t)\|_{L_{x}^{2}}^{p-2}\big(\|{\bm{f}}\|_{H_{x}^{-1}}^{2}+\|{\bm{g}}\|_{L_{x}^{2}}^{2}\big)\vd t.\nonumber 
\end{align}
On the other hand, by the Burkholder--Davis--Gundy
(BDG) inequality (cf. \cite{revuz1999continuous}), we can derive
from (\ref{eq:lpl2-1}) that 
\begin{align}
 & \E\sup_{t\in[0,\tau]}\|{\bm{u}}(t)\|_{L_{x}^{2}}^{p}+\E\int_{0}^{\tau}\|{\bm{u}}(t)\|_{L_{x}^{2}}^{p-2}\|\pd{\bm{u}}(t)\|_{L_{x}^{2}}^{2}\vd t\nonumber \\
\le\, & C\E\int_{0}^{\tau}\Big[\|{\bm{u}}(t)\|_{L_{x}^{2}}^{p}+\|{\bm{u}}(t)\|_{L_{x}^{2}}^{p-2}\big(\|{\bm{f}}\|_{H_{x}^{-1}}^{2}+\|{\bm{g}}\|_{L_{x}^{2}}^{2}\big)\Big]\vd t\label{eq:LPL2-3}\\
 & +C\E\bigg\{\int_{0}^{\tau}\|{\bm{u}}\|_{L_{x}^{2}}^{2(p-2)}\sum_{k}\bigg[\int_{\R^{d}}\big(\sigma_{\al\be}^{ik}u^{\al}\pd_{i}u^{\be}+u^{\al}g_{\al}^{k}\big)\vd x\bigg]^{2}\vd t\bigg\}^{\frac{1}{2}},\nonumber 
\end{align}
and by H\"{o}lder's inequality, the last term on the right-hand side of
the above inequality is dominated by 
\begin{align*}
 & C\,\E\bigg[\int_{0}^{\tau}\|{\bm{u}}\|_{L_{x}^{2}}^{2(p-2)}\big(\|{\bm{u}}\|_{L_{x}^{2}}^{2}\|\pd{\bm{u}}\|_{L_{x}^{2}}^{2}+\|{\bm{u}}\|_{L_{x}^{2}}^{2}\|\bm{g}\|_{L_{x}^{2}}^{2}\big)\vd t\bigg]^{\frac{1}{2}}\\
 & \le C\,\E\bigg\{\sup_{t\in[0,\tau]}\|{\bm{u}}(t)\|_{L_{x}^{2}}^{p/2}\bigg[\int_{0}^{\tau}\big(\|{\bm{u}}\|_{L_{x}^{2}}^{p-2}\|\pd{\bm{u}}\|_{L_{x}^{2}}^{2}+\|{\bm{u}}\|_{L_{x}^{2}}^{p-2}\|{\bm{g}}\|_{L_{x}^{2}}^{2}\big)\vd t\bigg]^{\frac{1}{2}}\bigg\}\\
 & \le\frac{1}{2}\,\E\sup_{t\in[0,\tau]}\|{\bm{u}}(t)\|_{L_{x}^{2}}^{p}+C\E\int_{0}^{\tau}\|{\bm{u}}\|_{L_{x}^{2}}^{p-2}\|\pd{\bm{u}}\|_{L_{x}^{2}}^{2}\vd t+C\int_{0}^{\tau}\|{\bm{u}}\|_{L_{x}^{2}}^{p-2}\|{\bm{g}}\|_{L_{x}^{2}}^{2}\vd t,
\end{align*}
which along with (\ref{eq:Lpl2-2}) and (\ref{eq:LPL2-3}) yields
that 
\begin{align*}
\E\sup_{t\in[0,\tau]}\|{\bm{u}}(t)\|_{L_{x}^{2}}^{p} & \leq C\E\int_{0}^{\tau}\|{\bm{u}}(t)\|_{L_{x}^{2}}^{p-2}\big(\|{\bm{f}}\|_{H_{x}^{-1}}^{2}+\|{\bm{g}}\|_{L_{x}^{2}}^{2}\big)\vd t\\
 & \leq C\E\bigg[\sup_{t\in[0,\tau]}\|{\bm{u}}(t)\|_{L_{x}^{2}}^{p-2}\int_{0}^{\tau}\big(\|{\bm{f}}\|_{H_{x}^{-1}}^{2}+\|{\bm{g}}\|_{L_{x}^{2}}^{2}\big)\vd t\bigg]\\
 & \leq\frac{1}{2}\E\sup_{t\in[0,\tau]}\|{\bm{u}}(t)\|_{L_{x}^{2}}^{p}+C\E\bigg[\int_{0}^{T}\big(\|{\bm{f}}\|_{H_{x}^{-1}}^{2}+\|{\bm{g}}\|_{L_{x}^{2}}^{2}\big)\vd t\bigg]^{\frac{p}{2}}.
\end{align*}
Thus we gain the estimate 
\begin{equation}
\frac{1}{C}\,\E\sup_{t\in[0,\tau]}\|{\bm{u}}(t)\|_{L_{x}^{2}}^{p}\leq\E\bigg[\int_{0}^{T}\big(\|{\bm{f}}\|_{H_{x}^{-1}}^{2}+\|{\bm{g}}\|_{L_{x}^{2}}^{2}\big)\vd t\bigg]^{\frac{p}{2}}=:F\label{eq:LPL2-4}
\end{equation}
with $C=C(d,\kappa,K,p,T)$.

In order to estimate $\|\pd\bm{u}_{x}\|_{L_{\om}^{p}L_{t}^{2}L_{x}^{2}}$,
we go back to (\ref{eq:l2-ito-2}). Bearing in mind Condition (\ref{eq:lp-condition})
(actually here we only need the weaker one (\ref{eq:p2}))
we can easily get that 
\begin{align*}
\|{\bm{u}}(\tau)\|_{L_{x}^{2}}^{2}+\epsilon\int_{0}^{\tau}\|\pd{\bm{u}}(t)\|_{L_{x}^{2}}^{2}\vd t & \le\int_{0}^{\tau}\int_{\R^{d}}\big(2u^{\al}f_{\al}+2\si_{\al\be}^{ik}\pd_{i}u^{\be}g_{\al}^{k}+|{\bm{g}}|^{2}\big)\vd x\vd t\\
 & \quad+\int_{0}^{\tau}\int_{\R^{d}}2\big(\sigma_{\al\be}^{ik}u^{\al}\pd_{i}u^{\be}+u^{\al}g_{\al}^{k}\big)\vd x\vd\BM_{t}^{k},
\end{align*}
where $\epsilon$ is the constant in Lemma \ref{lem:LH}.
Computing $\E[\,\cdot\,]^{p/2}$ on both sides of the above inequality
and by H\"{o}lder's inequality and the BDG inequality, we derive that
\begin{align*}
 & \E\biggl(\int_{0}^{\tau}\|\pd{\bm{u}}(t)\|_{L_{x}^{2}}^{2}\vd t\biggr)^{\!\frac{p}{2}}\\
\le\, & \frac{1}{4}\E\biggl(\int_{0}^{\tau}\|{\bm{u}}(t)\|_{H_{x}^{1}}^{2}\vd t\biggr)^{\!\frac{p}{2}}+CF+C\E\bigg|\int_{0}^{\tau}\int_{\R^{d}}\big(\sigma_{\al\be}^{ik}u^{\al}\pd_{i}u^{\be}+u^{\al}g_{\al}^{k}\big)\vd x\vd\BM_{t}^{k}\bigg|^{\frac{p}{2}}\\
\le\, & \frac{1}{4}\E\biggl(\int_{0}^{\tau}\|{\bm{u}}(t)\|_{H_{x}^{1}}^{2}\vd t\biggr)^{\!\frac{p}{2}}+CF+C\E\bigg[\sum_{k}\int_{0}^{\tau}\bigg\{\int_{\R^{d}}\big(\sigma_{\al\be}^{ik}u^{\al}\pd_{i}u^{\be}+u^{\al}g_{\al}^{k}\big)\vd x\bigg\}^{\!2}\!\md t\bigg]^{\!\frac{p}{4}}\\
\le\, & \frac{1}{4}\E\biggl(\int_{0}^{\tau}\|{\bm{u}}(t)\|_{H_{x}^{1}}^{2}\vd t\biggr)^{\!\frac{p}{2}}+CF+C\,\E\bigg[\int_{0}^{\tau}\|{\bm{u}}(t)\|_{L_{x}^{2}}^{2}\big(\|\pd{\bm{u}}(t)\|_{L_{x}^{2}}^{2}+\|{\bm{g}}(t)\|_{L_{x}^{2}}^{2}\big)\vd t\bigg]^{\!\frac{p}{4}}\\
\le\, & \frac{1}{2}\E\biggl(\int_{0}^{\tau}\|\pd{\bm{u}}(t)\|_{L_{x}^{2}}^{2}\vd t\biggr)^{\!\frac{p}{2}}+C\,\E\!\sup_{t\in[0,\tau]}\!\|{\bm{u}}(t)\|_{L_{x}^{2}}^{p}+CF.
\end{align*}
which along with (\ref{eq:LPL2-4}) implies 
\[
\E\sup_{t\in[0,\tau]}\|{\bm{u}}(t)\|_{L_{x}^{2}}^{p}+\E\biggl(\int_{0}^{\tau}\|\pd{\bm{u}}(t)\|_{L_{x}^{2}}^{2}\vd t\biggr)^{\frac{p}{2}}\le CF,
\]
where the constant $C$ depends only on $d,p,T,\kappa,$ and $K$,
but is independent of $\tau$. Finally, we take the stopping time
$\tau$ to be 
\[
\tau_{n}:=\inf\bigg\{ s\ge0:\sup_{t\in[0,s]}\|{\bm{u}}(t)\|_{L_{x}^{2}}^{2}+\int_{0}^{s}\|\pd{\bm{u}}(t)\|_{L_{x}^{2}}^{2}\vd t\ge n\bigg\}\wedge T,
\]
and letting $n$ tend to infinity we obtain the estimate (\ref{eq:LPL2-est})
with $m=0$. Theorem \ref{thm:globe-lpl2} is proved. 
\end{proof}

\section{\label{sec:Interior-H=0000F6lder-estimates}Interior H\"{o}lder estimates
for the model system}

The aim of this section is to prove the interior H\"{o}lder estimates
for the model equation \eqref{eq:model}. The conditions \eqref{eq:lp-condition}
and \eqref{eq:modeq-coe-bound} are also assumed throughout this section.
Take $\bm{f}\in C_{x}^{0}(\R^{d}\times\R;L_{\omega}^{p})$ and $\bm{g}\in C_{x}^{1}(\R^{d}\times\R;L_{\omega}^{p})$
such that the modulus of continuity
\[
\vp(r):=\esssup_{t\in\R,\,|x-y|\le r}(\|\bm{f}(x,t)-\bm{f}(y,t)\|_{L_{\om}^{p}}+\|\pd\bm{g}(x,t)-\pd\bm{g}(y,t)\|_{L_{\om}^{p}})
\]
satisfies the Dini condition:
\[
\int_{0}^{1}\frac{\vp(r)}{r}\vd r<\infty.
\]

\begin{thm}
\label{thm:basic}Let $\bm{u}\in\cfun_{x,t}^{2,1}(Q_{2};L_{\om}^{p})$
satisfy \eqref{eq:model}. Under the above setting, there is a positive
constant $C$, depending only on $d,\kappa,$ and $p$, such that
for any $X,Y\in Q_{1/4}$,
\[
\|\pd^{2}\bm{u}(X)-\pd^{2}\bm{u}(Y)\|_{L_{\omega}^{p}}\le C\,\bigg[\Delta M+\int_{0}^{\Delta}\frac{\vp(r)}{r}\vd r+\Delta\int_{\Delta}^{1}\frac{\vp(r)}{r^{2}}\vd r\bigg],
\]
where $\Delta:=|X-Y|_{{\rm p}}$ and 
\[
M:=\|\bm{u}\|_{L_{\omega}^{p}L_{t}^{2}L_{x}^{2}(Q_{1})}+\bbr\bm{f}\bbr_{0,p;Q_1}+\bbr\bm{g}\bbr_{1,p;Q_{1}}.
\]
\end{thm}
Then the interior H\"{o}lder estimate are straightforward:
\begin{cor}
\label{cor:holder-model}Under the same setting of Theorem \eqref{thm:basic}
and given $\hold\in(0,1)$, there is a constant $C>0$, depending
only on $d,\kappa$ and $p$, such that
\[
\lbr\pd^{2}\bm{u}\rbr_{(\hold,\hold/2),p;Q_{1/4}}\le C\bigg[\|\bm{u}\|_{L_{\omega}^{p}L_{t}^{2}L_{x}^{2}(Q_{1})}+\frac{\bbr\bm{f}\bbr_{\hold,p;Q_{1}}+\bbr\bm{g}\bbr_{1+\hold,p;Q_{1}}}{\hold(1-\hold)}\bigg],
\]
provided the right-hand side is finite.
\end{cor}
\begin{proof}[Proof of Theorem \ref{thm:basic}]
Letting $\vf:\R^n \to \R$ be a nonnegative and symmetric mollifier  and 
$\vf^\eps(x)=\eps^n\vf(x/\eps)$, we 
define $u^{\alpha,\eps} = \vf^\eps * u^\alpha$, $f^\eps_\alpha = \vf^\eps * f_\alpha$ and $g_\alpha^\eps = \vf^\eps * g_\alpha$.
Then it is easily checked that $\bm{f}^\eps$ and $\partial \bm{g}^\eps$ are also Dini continuous and has the same continuity modulus  $\vp$ with $\bm{f}$ and $\partial \bm{g}$,
and
\begin{gather*}
\bbr \bm{f}^\eps - \bm{f} \bbr_{0,p;\R^n} + \bbr \bm{g}^\eps - \bm{g} \bbr_{1,p;\R^n} \to 0,\\
\|\partial^2 \bm{u}^\eps(X)-\partial^2 \bm{u}(X)\|_{L^p_\omega} \to 0,
\quad \forall\,X\in\R^n \times \R,
\end{gather*}
as $\eps\to 0$.
On the other hand, from Fubini's theorem one can check that $\bm{u}^\eps$ satisfies the model equation~\eqref{eq:model} in the classical sense with free terms $\bm{f}^\eps$ and $\bm{g}^\eps$.
Therefore, it suffices to prove the theorem for the mollified functions,
and the general case is straightforward by passing the limits.

Based on the above analysis and the smoothness of mollified functions, we may suppose that (cf. \cite{du2015cauchy}) 
\begin{itemize}
\item[({\bf A})] 
$\bm{f},\bm{g}\in L^p_\omega L^2_t H^k_x (Q_R)\cap C^k_x(Q_R;L^p_\omega)$ for all $k\in\Nat$ and $R>0$.
\end{itemize}

We can also set $X=0$ without loss of generality. With $\rho=1/2$,
we define
\[
Q^{\m}:=Q_{\rho^{\m}}=Q_{\rho^{\m}}(0,0),\quad\m\in\mathbf{N}=\{0,1,2,\dots\},
\]
and introduce the following boundary value problems:
\begin{equation}
\bigg\{\begin{aligned}\md u^{\alpha,\m} & =\big[a_{\alpha\beta}^{ij}\partial_{ij}u^{\beta,\m}+f_{\alpha}(0,t)\big]\md t+\big[\sigma_{\alpha\beta}^{ik}\partial_{i}u^{\beta,\m}+g_{\alpha}^{k}(0,t)+x^{i}\pd_{i}g_{\alpha}^{k}(0,t)\big]\vd\BM_{t}^{k}\\
u^{\al,\m} & =u^{\al}\quad\text{on}\ \pd_{{\rm p}}Q^{\m},
\end{aligned}
\label{eq:appr}
\end{equation}
where $\pd_{{\rm p}}Q^{\m}$ denotes the parabolic boundary of the
cylinder $Q^{\m}$. The existence and interior regularity of $\bm{u}^{\m}$
can be direct yielded by Proposition \ref{lem:dirich}.

Given a point $Y=(y,s)\in Q_{1/4}$, there is an $\m_{0}\in\mathbf{N}$
such that
\[
\Delta:=|Y|_{{\rm p}}\in[\rho^{\m_{0}+2},\rho^{\m_{0}+1}).
\]
So we have
\begin{align}
 & \|\pd^{2}\bm{u}(Y)-\pd^{2}\bm{u}(0)\|_{L_{\om}^{p}}\label{eq:decomp}\\
\le\, & \|\pd^{2}\bm{u}^{\m_{0}}(0)-\pd^{2}\bm{u}(0)\|_{L_{\om}^{p}}+\|\pd^{2}\bm{u}^{\m_{0}}(Y)-\pd^{2}\bm{u}(Y)\|_{L_{\om}^{p}}+\|\pd^{2}\bm{u}^{\m_{0}}(Y)-\pd^{2}\bm{u}^{\m_{0}}(0)\|_{L_{\om}^{p}}\nonumber \\
=:\, & N_{1}+N_{1}'+N_{2}.\nonumber 
\end{align}
As $N_{1}$ and $N_{1}'$ are similar, we are going to derive the
estimates for $N_{1}$ and $N_{2}$. 
\begin{claim}
\label{lem:5-2}${\displaystyle \bbr\pd^{m}(\bm{u}^{\m}-\bm{u}^{\m+1})\bbr_{0,p;Q^{\m+2}}\le C(d,\kappa,p)\rho^{(2-m)\m-m}\vp(\rho^{\m})}$,
where $m\in\mathbf{N}$.
\end{claim}
\begin{proof}
Applying Proposition \ref{lem:local-LPL2-1} to \eqref{eq:appr},
we have
\[
\bbr\pd^{m}(\bm{u}^{\m}-\bm{u}^{\m+1})\bbr_{0,p;Q^{\m+2}}\le C\rho^{-m\m-m}\left\Vert \fint_{Q^{\m+1}}|\bm{u}^{\m}-\bm{u}^{\m+1}|^{2}\right\Vert _{L_{\omega}^{p/2}}^{1/2}=:I_{\m,m}
\]
(hereafter we denote $\fint_Q  = {1 \over |Q|}\int_Q$ with $|Q|$ being the Lebesgue measure of the set $Q\subset \R^{n+1}$), and by Proposition \ref{lem:dirich},
\[
J_{\m}:=\left\Vert \fint_{Q^{\m+1}}|\bm{u}^{\m}-\bm{u}|^{2}\right\Vert _{L_{\omega}^{p/2}}^{1/2}\le C\rho^{2\m}\vp(\rho^{\m}).
\]
So we gain that
\[
I_{\m,m}\le C\rho^{-m\m-m}(J_{\m}+J_{\m+1})\le C\rho^{(2-m)\m-m}\vp(\rho^{\m}).
\]
The claim is proved.
\end{proof}
\begin{claim}
\label{lem:5-3}${\displaystyle N_{1}\le C(d,\kappa,p)\int_{0}^{\rho^{\m_{0}}}\frac{\vp(r)}{r}\vd r}$.
\end{claim}
\begin{proof}
It follows from Claim \ref{lem:5-2} that
\[
\sum_{\m\ge\m_{0}}\|\pd^{2}\bm{u}^{\m}(0)-\pd^{2}\bm{u}^{\m+1}(0)\|_{L_{\om}^{p}}\le C\sum_{\m\ge\m_{0}}\vp(\rho^{\m})\le C\int_{0}^{\rho^{\m_{0}}}\frac{\vp(r)}{r}\vd r,
\]
which implies that $\pd^{2}\bm{u}^{\m}(0)$ converges in $L_{\om}^{p}$
as $\m\to\infty$, if the limit is $\pd^{2}\bm{u}(0)$, then
\[
N_{1}=\|\pd^{2}\bm{u}^{\m_{0}}(0)-\pd^{2}\bm{u}(0)\|_{L_{\om}^{p}}\le\sum_{\m\ge\m_{0}}\|\pd^{2}\bm{u}^{\m}(0)-\pd^{2}\bm{u}^{\m+1}(0)\|_{L_{\om}^{p}}\le C\int_{0}^{\rho^{\m_{0}}}\frac{\vp(r)}{r}\vd r.
\]
So it suffices to show that $\lim_{\m\to\infty}\|\pd^{2}\bm{u}^{\m}(0)-\pd^{2}\bm{u}(0)\|_{L_{\om}^{2}}=0$.
From Proposition \ref{lem:local-LPL2-1} with $p=2$, we have
\begin{align}\label{eq:Claim44-1}
&\sup_{Q^{\m+1}}\|\pd^{2}\bm{u}^{\m}-\pd^{2}\bm{u}\|^{2}_{L_{\om}^{2}}
\le C\rho^{-4\m}\E\fint_{Q^{\m}}|\bm{u}^{\m}-\bm{u}|^{2}
+C\,\E\fint_{Q^{\m}}\Big(|\bm{f}(x,t)-\bm{f}(0,t)|^{2}\\
&\qquad\qquad
+\rho^{-2l}|\bm{g}(x,t)-\bm{g}(0,t)-x^i\pd_i\bm{g}(0,t)|^{2}
+|\pd\bm{g}(x,t)-\pd\bm{g}(0,t)|^{2}\Big)\vd X\nonumber \\ 
&\qquad\qquad+ C\sum_{k=1}^{[\frac{d+1}{2}]+1}\rho^{2\m k}\E\fint_{Q^{\m}}\big(|\pd^{k}\bm{f}|^{2}+|\pd^{k+1}\bm{g}|^{2}\big). \nonumber
\end{align}
The additional assumption ({\bf A}) on $\bm{f}$ and $\bm{g}$ together with Proposition \ref{lem:dirich} implies
\begin{align*}
 & \rho^{-4\m}\E\fint_{Q^{\m}}|\bm{u}^{\m}-\bm{u}|^{2}\\
 & \le C\,\E\fint_{Q^{\m}}\big(|\bm{f}(x,t)-\bm{f}(0,t)|^{2}+\rho^{-2\m}|\bm{g}(x,t)-\bm{g}(0,t)-x^{i}\pd_{i}\bm{g}(0,t)|^{2}\big)\vd X\\
 & \le C\vp(\rho^{\m})^2\to0,\quad\text{as}\ \m\to\infty.
\end{align*}
 And it is easier to obtain that
the last two terms on the right-hand side of \eqref{eq:Claim44-1} tend to zero as $\m\to\infty$.
Thus, $\lim_{\m\to\infty}\|\pd^{2}\bm{u}^{\m}(0)-\pd^{2}\bm{u}(0)\|_{L_{\om}^{2}}=0.$
The claim is proved.
\end{proof}
\begin{claim}
\label{claim:5-4}${\displaystyle N_{2}\le C(d,\kappa,p)\rho^{\m_{0}}\bigg(M+\int_{\rho^{\m_{0}}}^{1}\frac{\vp(r)}{r^{2}}\vd r\bigg)}$.
\end{claim}
\begin{proof}
Define $\bm{h}^{\m}=\bm{u}^{\m}-\bm{u}^{\m-1}$ for $\m=1,2,\dots,\m_{0}$,
then
\begin{align*}
N_{2} & =\|\pd^{2}\bm{u}^{\m_{0}}(Y)-\pd^{2}\bm{u}^{\m_{0}}(0)\|_{L_{\om}^{p}}\\
 & \le\|\pd^{2}\bm{u}^{0}(Y)-\pd^{2}\bm{u}^{0}(0)\|_{L_{\om}^{p}}+\sum_{\m=1}^{\m_{0}}\|\pd^{2}\bm{h}^{\m}(Y)-\pd^{2}\bm{h}^{\m}(0)\|_{L_{\om}^{p}}.
\end{align*}
As $\pd_{ij}\bm{u}^{0}$ satisfies a homogeneous system in $Q_{1}$
for any $i,j=1,\dots,d$, it follows from Proposition \ref{lem:local-LPL2-1}
that, for $m=1,2$,
\begin{align*}
\bbr\pd^{m}(\pd_{ij}\bm{u}^{0})\bbr_{0,p;Q_{1/4}} &
  \le  C\|\pd_{ij}\bm{u}^{0}\|_{L_{\omega}^{p}L_{t}^{2}L_{x}^{2}(Q_{1/2})}\\
  & \le  C(\|\pd_{ij}\bm{u}^{0}-\pd_{ij}\bm{u}\|_{L_{\omega}^{p}L_{t}^{2}L_{x}^{2}(Q_{1/2})}+\|\pd_{ij}\bm{u}\|_{L_{\omega}^{p}L_{t}^{2}L_{x}^{2}(Q_{1/2})})\\
  &  \le  C(\|\bm{u}\|_{L_{\omega}^{p}L_{t}^{2}L_{x}^{2}(Q_{1})}+\bbr\bm{f}\bbr_{0,p;Q_1}+\bbr\bm{g}\bbr_{1,p;Q_{1}})
  = CM,
\end{align*}
and for $-1/16<s<t\le0$ and $x\in B_{1/4}$,
\begin{align*}
 \|\pd^{2}u^{\alpha,0}(x,t)-\pd^{2}u^{\alpha,0}(x,s)\|_{L_{\om}^{p}}
 &  =\left\Vert \int_{s}^{t}a_{\alpha\beta}^{ij}\partial_{ij}(\pd^{2}u^{\beta,0})\vd\tau+\int_{s}^{t}\sigma_{\alpha\beta}^{ik}\partial_{i}(\pd^{2}u^{\beta,0})\vd w_{\tau}^{k}\right\Vert _{L_{\om}^{p}}\\
 & \le C\sqrt{t-s}(\bbr\pd^{3}\bm{u}^{0}\bbr_{0,p;Q_{1/4}}+\bbr\pd^{4}\bm{u}^{0}\bbr_{0,p;Q_{1/4}})\\
 &  \le CM\sqrt{t-s}.
\end{align*}
So combining above two inequalities we have
\[
\|\pd^{2}\bm{u}^{{0}}(Y)-\pd^{2}\bm{u}^{{0}}(0)\|_{L_{\om}^{p}}\le CM|Y|_{{\rm p}}\le CM\rho^{\m_{0}}.
\]
Next, by Claim \ref{lem:5-2},
\[
\rho^{-\m}\bbr\pd^{3}\bm{h}^{\m}\bbr_{0,p;Q^{\m+1}}+\bbr\pd^{4}\bm{h}^{\m}\bbr_{0,p;Q^{\m+1}}\le C\rho^{-2\m}\vp(\rho^{\m-1}),
\]
thus, for $-\rho^{2(\m_{0}+1)}\le t\le0$ and $|x|\le\rho^{\m_{0}+1}$,
\[
\|\pd^{2}\bm{h}^{\m}(x,0)-\pd^{2}\bm{h}^{\m}(0,0)\|_{L_{\om}^{p}}\le C\rho^{\m_{0}-\m}\vp(\rho^{\m-1})
\]
and
\begin{align*}
  \|\pd^{2}h^{\alpha,\m}(x,t)-\pd^{2}h^{\alpha,\m}(x,0)\|_{L_{\om}^{p}}
  &  =\left\Vert \int_{0}^{t}a_{\alpha\beta}^{ij}\partial_{ij}(\pd^{2}h^{\beta,\m})\vd\tau+\int_{s}^{t}\sigma_{\alpha\beta}^{ik}\partial_{i}(\pd^{2}h^{\beta,\m})\vd w_{\tau}^{k}\right\Vert _{L_{\om}^{p}}\\
  & \le C(\rho^{\m_{0}}\bbr\pd^{3}\bm{h}^{\m}\bbr_{0,p;Q_{1/4}}+\rho^{2\m_{0}}\bbr\pd^{4}\bm{h}^{\m}\bbr_{0,p;Q_{1/4}})\\
  &  \le  C\rho^{\m_{0}-\m}\vp(\rho^{\m-1}).
\end{align*}
Therefore,
\[
N_{2}\le CM\rho^{\m_{0}}+C\sum_{\m=1}^{\m_{0}}\rho^{\m_{0}-\m}\vp(\rho^{\m-1})\le C\rho^{\m_{0}}\bigg(M+\int_{\rho^{\m_{0}}}^{1}\frac{\vp(r)}{r^{2}}\vd r\bigg).
\]
The claim is proved.
\end{proof}
Combining \eqref{eq:decomp} and Claims \ref{lem:5-3} and \ref{claim:5-4},
we conclude Theorem \ref{thm:basic}.
\end{proof}

\section{\label{sec:H=0000F6lder-estimates-for}H\"{o}lder estimates for general
systems}

This section is devoted to the proofs of Theorems \ref{thm:interior} and \ref{thm:cauchy}. We need two technical lemmas whose proofs can
be found in, for example, \cite{du2015cauchy}.
\begin{lem}
\label{lem:iter}Let $\vf:[0,T]\to[0,\infty)$ satisfy
\[
\vf(t)\le\tht\vf(s)+\sum_{i=1}^{m}A_{i}(s-t)^{-\eta_{i}}\quad\forall\,0\le t<s\le T
\]
for some nonnegative constants $\tht,\eta_{i}$ and $A_{i}$ ($i=1,\dots m$),
where $\tht<1$. Then
\[
\vf(0)\le C\sum_{i=1}^{m}A_{i}T^{-\eta_{i}},
\]
where $C$ depends only on $\eta_{1},\dots,\eta_{n}$ and $\tht$.
\end{lem}
\begin{lem}
\label{lem:interp}Let $p\ge1$, $R>0$ and $0\le s<r$. There exists a constant
$C>0$, depending only on $d$ and $p$, such that
\[
\lbr\bm{u}\rbr_{s,p;Q_{R}}\le C\eps^{r-s}\lbr\bm{u}\rbr_{r,p;Q_{R}}+C\eps^{-s-d/2}\big[\E\|\bm{u}\|_{L^{2}(Q_{R})}^{p}\big]^{\!\frac{1}{p}}
\]
for any $\bm{u}\in C^{r}(Q_{R};L_{\om}^{p})$ and $\eps\in(0,R)$.
\end{lem}

Now we prove the a priori \emph{interior H\"{o}lder estimates} for system \eqref{eq:main}.

\begin{proof}[Proof of Theorem~\ref{thm:interior}]
With a change of variable, we may move the point $X$ to the origin.
Let $\rho/2\le r<R\le\rho$ with $\rho\in(0,1/4)$ to be defined.
Take a nonnegative cut-off function $\zeta\in C_{0}^{\infty}(\R^{d+1})$
such that $\zeta=1$ on $Q_{r}$, $\zeta=0$ outside $Q_{R}$, and
for $\gamma\ge0$,
\[
[\zeta]_{(\gamma,\gamma/2);\R^{d+1}}\le C(d)\,(R-r)^{-\gamma}.
\]
Set $\bm{v}=\zeta\bm{u}$, and
\[
\tilde{a}_{\alpha\beta}^{ij}(t)=a_{\alpha\beta}^{ij}(0,t),\quad\tilde{\sigma}_{\alpha\beta}^{ik}(t)=\sigma_{\alpha\beta}^{ik}(0,t),
\]
then $\bm{v}=(v^{1},\dots,v^{N})$ satisfies
\[
\md v^{\alpha}=\big(\tilde{a}_{\alpha\beta}^{ij}\partial_{ij}v^{\beta}+\tilde{f}_{\alpha}\big)\md t+\big(\tilde{\sigma}_{\alpha\beta}^{ik}\partial_{i}v^{\beta}+\tilde{g}_{\alpha}^{k}\big)\vd\BM_{t}^{k}
\]
where
\begin{align*}
\tilde{f}_{\alpha} & =(a_{\alpha\beta}^{ij}-\tilde{a}_{\alpha\beta}^{ij})\zeta\partial_{ij}u^{\beta}+(b_{\alpha\beta}^{i}\zeta-2a_{\alpha\beta}^{ij}\pd_{j}\zeta)\pd_{i}u^{\beta}\\
 & \quad+(c_{\alpha\beta}\zeta-b_{\alpha\beta}^{i}\pd_{i}\zeta-a_{\alpha\beta}^{ij}\pd_{ij}\zeta)u^{\beta}+\zeta_{t}u^{\alpha}+\zeta f^{\alpha},\\
\tilde{g}_{\alpha}^{k} & =(\sigma_{\alpha\beta}^{ik}-\tilde{\sigma}_{\alpha\beta}^{ik})\zeta\partial_{i}u^{\beta}+(v^{k}\zeta-\sigma_{\alpha\beta}^{ik}\pd_{i}\zeta)u^{\beta}+\zeta g^{\alpha}.
\end{align*}
Obviously, $\tilde{a}_{\alpha\beta}^{ij}$ and $\tilde{\sigma}_{\alpha\beta}^{ik}$
satisfy the MSP condition with $\lambda=\lambda(0,t)$. So by Lemma
\ref{lem:interp},
\begin{align*}
\bbr\tilde{\bm{f}}\bbr_{\hold,p;Q_{R}} & \le(\eps+K\rho^{\hold})\lbr\pd^{2}\bm{u}\rbr_{\hold,p;Q_{R}}+C_{1}(R-r)^{-2-\hold-d/2}\|\bm{u}\|_{L_{\omega}^{p}L_{t}^{2}L_{x}^{2}(Q_{R})}\\
 & \quad+\lbr\bm{f}\rbr_{\hold,p;Q_{R}}+C_{1}(R-r)^{-\hold}\bbr\bm{f}\bbr_{0,p;Q_{R}},\\
\bbr\tilde{\bm{g}}\bbr_{1+\hold,p;Q_{R}} & \le(\eps+K\rho^{\hold})\lbr\bm{u}\rbr_{2+\hold,p;Q_{R}}+C_{1}(R-r)^{-2-\hold-d/2}\|\bm{u}\|_{L_{\omega}^{p}L_{t}^{2}L_{x}^{2}(Q_{R})}\\
 & \quad+\lbr\bm{g}\rbr_{1+\hold,p;Q_{R}}+C_{1}(R-r)^{-1-\hold}\bbr\bm{g}\bbr_{0,p;Q_{R}},
\end{align*}
where $C_{1}=C_{1}(d,K,p,\eps)$. Applying Corollary \ref{cor:holder-model},
we gain that 
\begin{align*}
 & \lbr\pd^{2}\bm{u}\rbr_{(\hold,\hold/2),p;Q_{r}}\\
 & \le C_{2}\big[(\eps+K\rho^{\hold})\,\lbr\pd^{2}\bm{u}\rbr_{(\hold,\hold/2),p;Q_{R}}+C_{1}(R-r)^{-2-\hold-d/2}\|\bm{u}\|_{L_{\omega}^{p}L_{t}^{2}L_{x}^{2}(Q_{R})}\\
 & \quad+\lbr\bm{f}\rbr_{\hold,p;Q_{R}}+C_{1}(R-r)^{-\hold}\bbr\bm{f}\bbr_{0,p;Q_{R}}+\lbr\bm{g}\rbr_{1+\hold,p;Q_{R}}+C_{1}(R-r)^{-1-\hold}\bbr\bm{g}\bbr_{0,p;Q_{R}}\big],
\end{align*}
where $C_{2}=C_{2}(d,\kappa,K,p,\hold)$. Set $\eps=(4C_{2})^{-1}$,
then 
\[
C_{2}(\eps+K\rho^{\hold})\le\frac{1}{2}\quad\text{for any}\ \rho\le(4C_{2}K)^{-1/\hold}=:\rho_{0}.
\]
Thus, by Lemma \ref{lem:iter} we have
\[
\lbr\pd^{2}\bm{u}\rbr_{(\hold,\hold/2),p;Q_{\rho/2}}\le C\big(\rho^{-2-\hold-d/2}\|\bm{u}\|_{L_{\omega}^{p}L_{t}^{2}L_{x}^{2}(Q_{\rho})}+\rho^{-\hold}\bbr\bm{f}\bbr_{\hold,p;Q_{\rho}}+\rho^{-1-\hold}\bbr\bm{g}\bbr_{1+\hold,p;Q_{\rho}}\big),
\]
where the constant $C$ depends only on $d,\kappa,K,p$, and $\hold$.
The proof is complete.
\end{proof}
\begin{proof}[Proof of Theorem \ref{thm:cauchy}]
The solvability of the Cauchy problem follows from the \emph{a priori}
estimate \eqref{eq:global} by the standard method of continuity (see
\cite[Theorem 5.2]{gilbarg2001elliptic}), so it suffices to prove
the \emph{a priori} estimate \eqref{eq:global}.

We may extend the equations to $\R^{d}\times(-\infty,T]\times\PS$
by letting $\bm{u}(x,t),\bm{f}(x,t)$ and $\bm{g}(x,t)$ be zero if
$t\le0$. Take $\tau\in(0,T]$ and $R=\rho_{0}/2$, where $\rho_{0}$
is determined in Theorem \ref{thm:interior}. 
Applying the estimate
\eqref{eq:interior} on the cylinders centered at $(x,s)$ for all
$s\in(-1,\tau],$ we can obtain that
\begin{align*}
\lbr\pd^{2}\bm{u}\rbr_{(\hold,\hold/2),p;\mathcal{Q}_{R,\tau}(x)} & \le C\big(\|\bm{u}\|_{L_{\omega}^{p}L_{t}^{2}L_{x}^{2}(\mathcal{Q}_{2R,\tau}(x))}+\bbr\bm{f}\bbr_{\hold,p;\mathcal{Q}_{2R,\tau}(x)}+\bbr\bm{g}\bbr_{1+\hold,p;\mathcal{Q}_{2R,\tau}(x)}\big)\\
 & \le C\big(\|\bm{u}\|_{L_{\omega}^{p}L_{t}^{2}L_{x}^{2}(\mathcal{Q}_{2R,\tau}(x))}+\bbr\bm{f}\bbr_{\hold,p;\mathcal{Q}_{\tau}}+\bbr\bm{g}\bbr_{1+\hold,p;\mathcal{Q}_{\tau}}\big),
\end{align*}
then by Lemma \ref{lem:interp},
\begin{align}
\bbr\bm{u}\bbr_{(2+\hold,\hold/2),p;\mathcal{Q}_{R,\tau}(x)} & \le C\big(\|\bm{u}\|_{L_{\omega}^{p}L_{t}^{2}L_{x}^{2}(\mathcal{Q}_{2R,\tau}(x))}+\bbr\bm{f}\bbr_{\hold,p;\mathcal{Q}_{\tau}}+\bbr\bm{g}\bbr_{1+\hold,p;\mathcal{Q}_{\tau}}\big).\label{eq:5-003}
\end{align}
Define
\begin{align*}
M_{x,R}^{\tau}(\bm{u}) =\sup_{0\le t\le\tau}\biggl(\fint_{B_{R}(x)}\E|\bm{u}(y,t)|^{p}\vd y\biggr)^{\frac{1}{p}},\qquad
M_{R}^{\tau}(\bm{u}) =\sup_{x\in\R^{d}}M_{x,R}^{\tau}(\bm{u}).
\end{align*}
Obviously, $\|\bm{u}\|_{L_{\omega}^{p}L_{t}^{2}L_{x}^{2}(\mathcal{Q}_{2R,\tau}(x))}\le C(d,p,R)\,M_{R}^{\tau}(\bm{u}).$
So \eqref{eq:5-003} implies 
\begin{align}
\sup_{x\in\R^{d}}\bbr\bm{u}\bbr_{(2+\hold,\hold/2),p;\mathcal{Q}_{R,\tau}(x)} & \le C_3\big(M_{R}^{\tau}(\bm{u})+\bbr\bm{f}\bbr_{\hold,p;\mathcal{Q}_{\tau}}+\bbr\bm{g}\bbr_{1+\hold,p;\mathcal{Q}_{\tau}}\big).\label{eq:5-004}
\end{align}
To get rid of $M_{R}^{\tau}(\bm{u})$, we apply It\^o's formula to
$|\bm{u}|^{p}$:
\begin{align*}
\md|\bm{u}|^{p} & =p|\bm{u}|^{p-2}\Big[u^{\alpha}(a_{\alpha\beta}^{ij}\pd_{ij}u^{\beta}+b_{\alpha\beta}^{ij}\pd_{i}u^{\beta}+c_{\alpha\beta}u^{\beta}+f_{\alpha})+\frac{1}{2}\sum_{k}(\si_{\al\be}^{ik}\pd_{i}u^{\be}+g_{\al}^{k})^{2}\Big]\vd t\\
 & \quad+\frac{p(p-2)}{2}\bm{1}_{\{|\bm{u}|\neq0\}}|\bm{u}|^{p-4}\sum_{k}(\si_{\al\be}^{ik}u^{\al}\pd_{i}u^{\be}+u^{\al}g_{\al}^{k})^{2}\vd t+\md M_{t},
\end{align*}
where $M_{t}$ is a martingale. Integrating on $\mathcal{Q}_{R,\tau}(x)\times\PS$
and by the H\"{o}lder inequality, we can derive that 
\[
\sup_{t\in[0,\tau]}\E\int_{B_{R}(x)}|\bm{u}(y,t)|^{p}\vd y\le C_{4}\,\E\int_{\mathcal{Q}_{R,\tau}(x)}(|\pd^{2}\bm{u}|^{p}+|\bm{u}|^{p}+|\bm{f}|^{p}+|\bm{g}|^{p})\vd X
\]
with $C_{4}=C_{4}(d,N,K,p)$, which implies that
\begin{align*}
M_{x,R}^{\tau}(\bm{u})
& \le C_{4}\tau\,\big(\bbr\bm{u}\bbr_{2,p;\mathcal{Q}_{R,\tau}(x)}+\bbr\bm{f}\bbr_{0,p;\mathcal{Q}_{\tau}}+\bbr\bm{g}\bbr_{0,p;\mathcal{Q}_{\tau}}\big) \\
& \le{C_{4}\tau}\Big(\sup_{x\in\R^{d}}\bbr\bm{u}\bbr_{(2+\hold,\hold/2),p;\mathcal{Q}_{R,\tau}(x)}+\bbr\bm{f}\bbr_{0,p;\mathcal{Q}_{\tau}}+\bbr\bm{g}\bbr_{0,p;\mathcal{Q}_{\tau}}\Big),
\end{align*}
Substituting the last relation into \eqref{eq:5-004} and taking $\tau=(2C_{3}C_{4})^{-1}$,
we get
\begin{align*}
\sup_{x\in\R^{d}}\bbr\bm{u}\bbr_{(2+\hold,\hold/2),p;\mathcal{Q}_{R,\tau}(x)} & \le C\big(\bbr\bm{f}\bbr_{\hold,p;\mathcal{Q}_{\tau}}+\bbr\bm{g}\bbr_{1+\hold,p;\mathcal{Q}_{\tau}}\big),
\end{align*}
and equivalently, 
\begin{align}
\bbr\bm{u}\bbr_{(2+\hold,\hold/2),p;\mathcal{Q}_{\tau}} & \le C_{(\tau)}\big(\bbr\bm{f}\bbr_{\hold,p;\mathcal{Q}_{\tau}}+\bbr\bm{g}\bbr_{1+\hold,p;\mathcal{Q}_{\tau}}\big)\label{eq:5-005}
\end{align}
with $C_{(\tau)}=C_{(\tau)}(d,\kappa,K,p,\hold)\ge1$.

Let us conclude the proof by induction. Assume that there is a constant
$C_{(S)}\ge1$ for some $S>0$ such that
\begin{align*}
\bbr\bm{u}\bbr_{(2+\hold,\hold/2),p;\mathcal{Q}_{S}} & \le C_{(S)}\big(\bbr\bm{f}\bbr_{\hold,p;\mathcal{Q}_{S}}+\bbr\bm{g}\bbr_{1+\hold,p;\mathcal{Q}_{S}}\big).
\end{align*}
Then applying \eqref{eq:5-005} to $\bm{v}(x,t)=\bm{1}_{\{t\ge0\}}\cdot[\bm{u}(x,t+S)-\bm{u}(x,S)]$,
one can easily derive that
\begin{align*}
\bbr\bm{v}\bbr_{(2+\hold,\hold/2),p;\mathcal{Q}_{\tau}} & \le C_{(\tau)}\big(\bbr\bm{f}\bbr_{\hold,p;\mathcal{Q}_{S+\tau}}+\bbr\bm{g}\bbr_{1+\hold,p;\mathcal{Q}_{S+\tau}}+\tilde{C}\bbr\bm{u}(\cdot,S)\bbr_{2+\hold,p;\R^{d}}\big)\\
 & \le2C_{(\tau)}\tilde{C}C_{(S)}\big(\bbr\bm{f}\bbr_{\hold,p;\mathcal{Q}_{S+\tau}}+\bbr\bm{g}\bbr_{1+\hold,p;\mathcal{Q}_{S+\tau}}\big),
\end{align*}
with $\tilde{C}=\tilde{C}(N,K)\ge1$, so
\begin{align*}
\bbr\bm{u}\bbr_{(2+\hold,\hold/2),p;\mathcal{Q}_{S+\tau}} & \le\bbr\bm{v}\bbr_{(2+\hold,\hold/2),p;\mathcal{Q}_{\tau}}+2\bbr\bm{u}\bbr_{(2+\hold,\hold/2),p;\mathcal{Q}_{S}}\\
 & \le4C_{(\tau)}\tilde{C}C_{(S)}\big(\bbr\bm{f}\bbr_{\hold,p;\mathcal{Q}_{S+\tau}}+\bbr\bm{g}\bbr_{1+\hold,p;\mathcal{Q}_{S+\tau}}\big),
\end{align*}
 that means $C_{(S+\tau)}\le4C_{(\tau)}\tilde{C}C_{(S)}$. By iteration we have $C_{S}\le C\me^{CS}$ with $C=C(d,N,\kappa,K,p,\hold)$,
and the theorem is proved.
\end{proof}

\section{More comments on the MSP condition}

In this section we discuss more examples on the sharpness and flexibility
of the MSP condition (Definition~\ref{cond:msp}). We always let $d=1$ and assume that
the coefficient matrices $A=[a_{\alpha\beta}]$ and $B=[\sigma_{\alpha\beta}]$
are \emph{constant}. We write $M\gg0$ if the matrix $M$ is positive
definite.

Under the above setting the MSP condition can be written into
the following form if we set [$\lambda_{\alpha\beta}^{ik}]=(B+B')/2-\Lambda$ in \eqref{eq:Aij}. 
\begin{condition}
\label{cond:case-d1}There is a symmetric $N\times N$ real matrix
$\Lambda$ such that
\begin{equation}
A+A'-B'B-(p-2)(T_{B}+\Lambda)'(T_{B}+\Lambda)\gg0\label{eq:d1}
\end{equation}
where $T_{B}:=(B-B')/2$ is the skew-symmetric component
of $B$.
\end{condition}

\begin{example}
consider the following system
\begin{equation}
\bigg\{\begin{aligned}\md u^{(1)} & =u_{xx}^{(1)}\vd t+(\lambda u_{x}^{(1)}-\mu u_{x}^{(2)})\vd\BM_{t},\\
\md u^{(2)} & =u_{xx}^{(2)}\vd t+(\mu u_{x}^{(1)}+\lambda u_{x}^{(2)})\vd\BM_{t}
\end{aligned}
\label{eq:exam}
\end{equation}
with $x\in\mathbf{T}=\R/(2\pi\mathbf{Z})$, real constants $\lambda$
and $\mu$, and with the initial data
\begin{equation}
u^{(1)}(x,0)+\sqrt{-1}u^{(2)}(x,0)=\sum\nolimits_{n\in\mathbf{Z}}\me^{-n^{2}}\cdot\me^{\sqrt{-1}nx}.\label{eq:initial}
\end{equation}
Evidently, if $\lambda^{2}+\mu^{2}<2$, then system \eqref{eq:exam}
satisfies the condition \eqref{eq:parabolic}, and from the result
of \cite{Kim2013w}, it has a unique solution $\bm{u}=(u^{(1)},u^{(2)})'$
in the space $L^{2}(\PS;C([0,T];H^{m}(\mathbf{T}))$ with any $m\ge0$
and $T>0$. 

To apply our results to \eqref{eq:exam}, we should assume it to satisfy
Condition \ref{cond:case-d1}. In the next two lemma, we first simplify
the condition into a specific constraint on $\lambda$ and $\mu$,
and then prove it to be optimal. 
\begin{lem}
Let $p\ge2$. The coefficients of system \eqref{eq:exam} satisfies
Condition \ref{cond:case-d1} if and only if they satisfy \eqref{eq:d1}
with $\Lambda=0$, namely,
\begin{equation}
\lambda^{2}+(p-1)\mu^{2}<2.\label{eq:6-2}
\end{equation}
\end{lem}
\begin{proof}
By orthogonal transform, $A+A'-B'B-(p-2)(T_{B}+\Lambda)'(T_{B}+\Lambda)$
is positive definite if and only if 
\begin{equation}
2-(\lambda^{2}+\mu^{2})-(p-2)\lambda_{\max}>0,\label{eq:6-3}
\end{equation}
where $\lambda_{\max}$ is the larger eigenvalue of $(T_{B}+\Lambda)'(T_{B}+\Lambda)$.
For $\Lambda=\mu\begin{bmatrix}a & c\\
c & b
\end{bmatrix}$, we have 
\[
(T_{B}+\Lambda)'(T_{B}+\Lambda)=\mu^{2}\begin{bmatrix}a^{2}+(c-1)^{2} & ac+bc+a-b\\
ac+bc+a-b & b^{2}+(c+1)^{2}
\end{bmatrix}
\]
whose larger eigenvalue is
\[
\lambda_{\max}=\frac{\mu^{2}}{2}(a^{2}+b^{2}+2c^{2}+2)+\frac{\mu^{2}}{2}\sqrt{(a^{2}-b^{2}-4c)^{2}+4(ac+bc+a-b)^{2}}.
\]
Obviously, $\lambda_{\max}\ge\mu^{2}$. 

Once \eqref{eq:6-3} holds for some $\Lambda$, we get \eqref{eq:6-2}, namely \eqref{eq:d1} holds for  $\Lambda=0$. 
 Now we prove the {\it only if} part. The proof of {\it if} part is trivial.
\end{proof}
Therefore, if \eqref{eq:6-2} is satisfied, then $\sup_{x\in\mathbf{T}}\E\|\bm{u}(x,t)\|^{p}<\infty$
for any $t\ge0$; if it is not, even some weaker norm of $\bm{u}(\cdot,t)$
is infinite for large $t$ as showed in the following lemma.
\begin{lem}
Let $p>2$ and $\lambda^{2}+\mu^{2}<2$. If $\eps:=\lambda^{2}+(p-1)\mu^{2}-2>0$,
then 
\[
\E\|\bm{u}(\cdot,t)\|_{L^{2}(\mathbf{T})}^{p}=\infty
\]
for any $t>2/\eps$.
\end{lem}
\begin{proof}
Denote $v=u^{(1)}+\sqrt{-1}u^{(2)}$ that can be verified to satisfy
\[
\md v=v_{xx}\vd t+(\lambda+\sqrt{-1}\mu)v_{x}\vd w_{t}
\]
with the initial condition $v(x,0)=\sum_{n\in\mathbf{Z}}\me^{-n^{2}}\me^{\sqrt{-1}nx}$
for $x\in\mathbf{T}$. By Fourier analysis, we can express
\[
v(x,t)=\sum\nolimits_{n\in\mathbf{Z}}v_{n}(t)\me^{\sqrt{-1}nx},
\]
where $v_{n}(\cdot)$ satisfies the following SDE:
\[
\md v_{n}=v_{n}[-n^{2}\vd t+(-\mu+\sqrt{-1}\lambda)n\vd w_{t}],\quad v_{n}(0)=\me^{-n^{2}}.
\]
From the theory of SDEs, we have
\[
v_{n}(t)=\me^{-\frac{1}{2}f(t)n^{2}-\mu nw_{t}}\cdot\me^{\sqrt{-1}(\lambda\mu n^{2}t+\lambda nw_{t})},
\]
where $f(t):=2+(2+\mu^{2}-\lambda^{2})t$. So we derive
\begin{align*}
|v_{n}(t)|^{2} & =\exp\!\left\{ -f(t)n^{2}-2\mu nw_{t}\right\} \\
 & =\exp\biggl\{-f(t)\left(n+\frac{\mu w_{t}}{f(t)}\right)^{\!2}+\frac{\mu^{2}|w_{t}|^{2}}{f(t)}\biggr\},
\end{align*}
and by Parseval's identity, 
\begin{align*}
\|v(\cdot,t)\|_{L^{2}(\mathbf{T})}^{2} & =2\pi\sum_{n\in\mathbf{Z}}|v_{n}(t)|^{2}\\
 & =2\pi\sum_{n\in\mathbf{Z}}\exp\biggl\{-f(t)\left(n+\frac{\mu w_{t}}{f(t)}\right)^{\!2}+\frac{\mu^{2}|w_{t}|^{2}}{f(t)}\biggr\}\\
 & \ge2\pi\exp\biggl\{-f(t)+\frac{\mu^{2}|w_{t}|^{2}}{f(t)}\biggr\}.
\end{align*}
Thus, we have
\begin{align*}
\E\|\bm{u}(\cdot,t)\|_{L^{2}(\mathbf{T})}^{p} & =\E\|v(\cdot,t)\|_{L^{2}(\mathbf{T})}^{p}\\
 & \ge(2\pi)^p\,\E\exp\biggl\{-\frac{pf(t)}{2}+\frac{p\mu^{2}|w_{t}|^{2}}{2f(t)}\biggr\}\\
 & =(2\pi)^p\me^{-pf(t)/2}\,\E\exp\biggl\{\frac{p\mu^{2}|w_{1}|^{2}}{2f(t)/t}\biggr\}\\
 & =(2\pi)^p\me^{-pf(t)/2}\,\E\exp\biggl\{\frac{p\mu^{2}|w_{1}|^{2}}{2[2+\mu^{2}-\lambda^{2}+2t^{-1}]}\biggr\}\\
 & =(2\pi)^{p-1/2}\me^{-pf(t)/2}\int_{\R}\exp\biggl\{-\frac{y^{2}}{2}\biggl[1-\frac{p\mu^{2}}{2+\mu^{2}-\lambda^{2}+2t^{-1}}\biggr]\biggr\}\md y.
\end{align*}
The last integral diverges if
\[
1-\frac{p\mu^{2}}{2+\mu^{2}-\lambda^{2}+2t^{-1}}<0.
\]
This immediately concludes the lemma.
\end{proof}
\end{example}

Indeed, some specific choices of $\Lambda$ in Condition~\ref{eq:d1}
like $\Lambda=0$ usually lead to a class of convenient and even optimal
criteria in applications. For instance, the above discussion shows
how the skew-symmetric component of $B$ substantially affects the
$L^{p}$-norm of the solution of system \eqref{eq:exam}. But in general,
the choice of $\Lambda$ still heavily depends on the structure of
the concrete problem. 

\begin{example}\label{lem:last}
Let $p\ge3$ and $\lambda>\mu>0$. 
Consider 
\[
A=\begin{bmatrix} 1+\lambda^2 & 0\\
0 & 1+\mu^2
\end{bmatrix}
\quad\text{and}\quad 
B=\begin{bmatrix}0 & -\mu\\
\lambda & 0
\end{bmatrix}.
\]
For the sake of simplicity, we restrict the choice of $\Lambda$
in the form $\begin{bmatrix}0 & c\\
c & 0
\end{bmatrix}$. Then we have
\begin{align*}
 & A+A'-B'B-(p-2)(T_{B}+\Lambda)'(T_{B}+\Lambda)\\
 & ={\rm diag}\Bigl\{2+\lambda^{2}-(p-2)\Big(c+\frac{\lambda+\mu}{2}\Big)^{\!2},\ 2+\mu^{2}-(p-2)\Big(c-\frac{\lambda+\mu}{2}\Big)^{\!2}\Bigr\}\\
 & =:{\rm diag}\{g(c),\,h(c)\}.
\end{align*}
As $p\ge3$ and $\lambda>\mu>0$, it is easily to check that 
\[
\max_{c\in\R}\big\{ g(c)\wedge h(c)\big\}
= 2 + {\lambda^2 + \mu^2 \over 2 }-\frac{(p-2)(\lambda+\mu)^{2}}{4}-\frac{(\lambda-\mu)^{2}}{4(p-2)},
\]
where the maximum is attained when $g(c)=h(c)$, i.e.,
\[
c=\frac{\lambda - \mu}{2(p-2)}.
\]
So one can easily assign some specific values to $p$, $\lambda$ and $\mu$ to let $A$ and $B$ satisfy Condition~\ref{eq:d1} but not with  $\Lambda=0$, for example, $(p,\lambda,\mu)=(3,3,1)$. 
This shows that the choice $\Lambda=0$ does not always lead to the minimal requirements.
\end{example}

\bibliographystyle{imsart-nameyear}
\bibliography{spds}

\end{document}